\documentclass[11pt]{article}

\usepackage{amsmath}
\usepackage{amsfonts}
\usepackage{amssymb}
\usepackage{amsmath}
\usepackage{amsthm}
\usepackage{latexsym}
\usepackage{graphicx,cite}
\usepackage{float}
\usepackage{a4wide}

\usepackage{color}
\definecolor{hanblue}{rgb}{0.27, 0.42, 0.81}
\definecolor{red}{rgb}{1.0, 0.0, 0.0}
\usepackage[colorlinks, citecolor=hanblue,linkcolor=red]{hyperref}

\newtheorem{thm}{Theorem}[section]
\newtheorem{lem}[thm]{Lemma}
\newtheorem{prop}[thm]{Proposition}
\newtheorem{cor}[thm]{Corollary}

\theoremstyle{definition}
\newtheorem{defn}[thm]{Definition}

\newtheorem{prob}[thm]{Problem}

\theoremstyle{remark}
\newtheorem{rem}[thm]{Remark}

\numberwithin{equation}{section}

\usepackage{tikz,fp,ifthen,fullpage}
\usepackage{pgfmath}
\usetikzlibrary{backgrounds}
\usetikzlibrary{decorations.pathmorphing,backgrounds,fit,calc,through}
\usetikzlibrary{arrows}
\usetikzlibrary{shapes,decorations,shadows}
\usetikzlibrary{fadings}
\usetikzlibrary{patterns}
\usetikzlibrary{mindmap}
\usetikzlibrary{decorations.text}
\usetikzlibrary{decorations.shapes}

\title{Minimal elastic networks}

\author{Anna Dall'Acqua
\footnote{Institut f\"{u}r Analysis, Fakult\"{a}t f\"{u}r Mathematik und Wirtschaftswissenschaften,
Universit\"{a}t Ulm, 89081 Ulm, Germany}
\and Matteo Novaga
\footnote{Dipartimento di Matematica, Universit\`{a} di Pisa, Largo Bruno Pontecorvo 5,
56127 Pisa, Italy}
\and Alessandra Pluda
\footnote{Fakult\"{a}t f\"{u}r Mathematik, Universit\"{a}t Regensburg, Universit\"{a}tstrasse 31,
93053 Regensburg, Germany}
}

\begin{document}

\maketitle

\begin{abstract}
We consider planar networks of three curves that meet at two junctions
with prescribed equal angles, minimizing a combination of the elastic energy
and the length functional. We prove existence and regularity of minimizers,
and we show some properties of the minimal configurations.
\end{abstract}

%
\section{Introduction}
In this paper we are interested in the minimization of the elastic energy among
planar networks composed by three curves which
we call $3$--networks. 
In this still large class we consider the more special family of, what we call, 
Theta--networks. These are 
$3$--networks whose curves are of class
$H^2$, regular and form $120$ degrees at the two junctions. 

We consider an energy 
given by the sum of the  penalized elastic energies of each 
curve composing the network.
Parametrizing each curve by $\gamma^i$,
denoting with
$s^i$ is the arclength parameter and  with
 $k^i$ 
the (scalar) curvature, we set
\begin{equation}\label{functionalalphagen}
F _\alpha\left(\Gamma\right):=\sum_{i=1}^3\int_{\gamma^i}\left( (k^i)^{2}+\alpha\right){\rm{d}}s^i\,,
\qquad \alpha>0.
\end{equation}
Although  we are interested in this energy mainly 
 from a theoretical point of view, we remind that
it appears also in many mechanical and physical models (c.f.~\cite{Truesdell}) and in imaging
sciences, see for instance~\cite{Mumford1994}.

\smallskip

Before embarking on the description of our work, 
let us briefly review the known results in 
the case of single closed curves. 
We remark that if we
take $\alpha=0$ in~\eqref{functionalalphagen}
the infimum of the elastic energy among closed curves
is zero, but it is not a minimum.
Indeed consider a sequence of circles $\mathcal{C}_R$ with radius $R$,
then the energy
vanishes as $R\to \infty$, but the value zero is never attained.
Adding a length penalization to the functional 
is needed  to have a well-posed problem.
Due to the behaviour under rescaling of the elastic energy and of the length
it is easy to see that, up to rescaling, 
one can reduce to consider the functional $F_1$ (see~\cite{daplu} and Section 2).
The minimum of the energy $F_1$ is $4\pi$ and 
the unique minimizer is the circle of radius $1$.

Stationary curves  for the functional $F_1$ (or more generally $F_\alpha$)
are usually called  \emph{elasticae}. 
In~\cite{langsing1} 
Langer and Singer provide a 
complete classification of  elasticae.
In particular, 
the authors prove that 
the circle and  the ``Figure Eight" 
(or a multiple cover of one of these two)
are the unique closed planar elasticae~\cite[Theorem 0.1 (a)]{langsing1}. 
Let us observe that minimizing the elastic energy on closed curves 
with a fixed length constrained is, 
up to rescaling, equivalent to minimizing the penalized elastic energy on the same class of curves, see~\cite{daplu} for the details.
Moreover we notice that penalizing the length is not the unique way
to obtain a well-posed and non-trivial minimization
problem, for instance one can
penalize the area enclosed by the curves~\cite{bucurhenrot, ferkawni} 
(in this case the curves are necessary embedded) 
or confine the curves
in a bounded domain of $\mathbb{R}^2$~\cite{domuro,masnou}. 

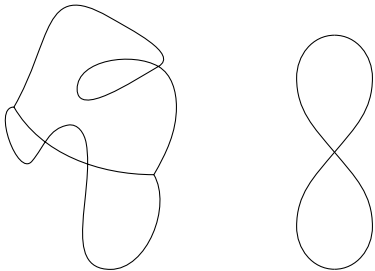
\begin{figure}[H]
\begin{center}
\begin{tikzpicture}[scale=1.2]
\draw[shift={(0,0)}] 
(-1.73,-1.8)
to[out= 180,in=-80, looseness=1] 
(-2,-0.4) 
to[out= 100,in=10, looseness=1] 
(-2.2,-0.2) 
to[out= -170,in=50, looseness=1] 
(-2.6,-0.6) 
to[out= -130,in=180, looseness=1] 
(-2.8,0) 
to[out= 60,in=150, looseness=1.5] (-1.75,1) 
(-2.8,0)
to[out=-60,in=180, looseness=0.9] (-1.25,-0.75)
(-1.75,1)
to[out= -30,in=30, looseness=0.9](-1.2,0.45)
(-2.1,0.2)
to[out= -90,in=-150, looseness=0.9] 
(-1.2,0.45)
(-2.1,0.2)
to[out= 90,in=150, looseness=0.9] 
(-1.2,0.45)
to[out=-30,in=90, looseness=0.9] (-1,0)
to[out= -90,in=60, looseness=0.9] (-1.25,-0.75)
to[out= -60,in=0, looseness=0.9](-1.73,-1.8);
\end{tikzpicture}\qquad\qquad
 \begin{tikzpicture}[scale=0.9]
\draw
(0,1.73)to[out= 0,in=90, looseness=1] (0.56,1.1)
(0.56,1.1)to[out= -90,in=50, looseness=1] (0,0);
\draw
(0,1.73)to[out= 180,in=90, looseness=1] (-0.56,1.1)
(-0.56,1.1)to[out= -90,in=130, looseness=1] (0,0);
\draw[rotate=180]
(0,1.73)to[out= 0,in=90, looseness=1] (0.56,1.1)
(0.56,1.1)to[out= -90,in=50, looseness=1] (0,0);
\draw[rotate=180]
(0,1.73)to[out= 180,in=90, looseness=1] (-0.56,1.1)
(-0.56,1.1)to[out= -90,in=130, looseness=1] (0,0);
\end{tikzpicture}
\end{center}
\caption{Left: A Theta--network. Right: The ``Figure Eight", that is, the unique
closed planar elastica with self-intersections.}\label{otto}
\end{figure}

Let us now come back to the networks' case.
Similarly to what happens with curves, 
if we state the minimization problem among regular $3$--networks of class $H^2$
again the infimum is zero and never attained.
To avoid such inconvenience
we restrict to the class of Theta--networks:
we ask that the unit tangent vectors to the three curves at the junctions form 
angles of $120$ degrees. In fact to find a lower bound on the energy
it would be enough to require that the angles are fixed and different from zero
(see also~\cite{daplu} for a long digression about 
the possible choices to get a well-posed problem).

\smallskip

Notice however that also in this class the problem has some form of degeneracy.
Indeed consider a minimizing sequence of Theta--networks $\Gamma_n$,
then as $n\to\infty$
the length of some curves could go to zero. 
More precisely we are able to prove that 
the length of at most one curve can go to zero.
The ($H^2$--weak) limit 
is either a Theta--network,
or it is composed by two drops 
(Definition~\ref{def:drop}) forming angles in pair of $120$ and $60$ degrees. 
A key point of the proof 
are careful estimates 
on the energy, 
based on a generalization
of the classical Gauss--Bonnet theorem (Theorem~\ref{stimamodulok}).

\smallskip

Being the class of Theta--networks not closed, 
we enlarge the class of networks into one where we
gain compactness.
We introduce the extension $\overline{F}$ of the original functional $F$
(Definition~\ref{funzionalerilassato}) defined for all $3$--networks
and we prove existence of minimizers for $\overline{F}$
(Corollary~\ref{esistenzatrenet}).

\medskip

This is clearly not enough to show existence of minimizers
for $F$ among Theta--network. To conclude the argument we prove that
$\overline{F}$ is the relaxation (w.r.t. the weak $H^2$ topology)
of $F$  and then we show that the minimizers of $\overline{F}$
actually belong to the initial class of Theta--networks.\\
The proof of the latter claim is quite tricky, and is one of the main results of this paper.
We briefly describe the structure of the argument.
We first establish the existence of a minimizer for the penalized elastic energy in the class of drops, and compute its energy. 
The unique minimizer (up to isometries of $\mathbb{R}^2$)
turns out to be  one of the two drops of 
the ``Figure Eight". Following the approach introduced in~\cite{explicitelasticae},
in~\cite{daplu} we show that the energy of the minimal drop
is approximately $10.60375$ (see Lemma~\ref{energiaotto}).  
Then we suppose by contradiction that a minimizer of $\overline{F}$ 
is composed by two drops.
Its energy is greater or equal than
the one of the ``Figure Eight". It is therefore enough to exhibit a 
Theta--network with strictly less energy 
to get the contradiction. The energy of the standard double bubble 
(easy to compute)  is strictly less  than the energy of the ``Figure Eight"
and this allows us to conclude the proof.



\medskip

The plan of the paper is the following:
In Section~\ref{basic} we introduce the penalized elastic energy functional for Theta--networks
and we state some basic properties of this energy.
Section~\ref{esistenza}  is devoted to prove existence of minimizers for the
relaxed functional. In Section~\ref{prop}
 we consider the minimization problem of the energy in the class
of drops. We describe
the unique minimizer (up to isometries of $\mathbb{R}^2$), whose energy is
computed in~\cite{daplu}.
Using this characterization
we are able to prove that
the minimizers of our problem for networks
are in the original smaller class of Theta--networks, and that
each curve of a minimizer is injective
(see Proposition~\ref{iniet}). We also shortly discuss the minimization problem for networks with two triple junctions and prescribed non-equal angles. We conclude the paper with an
estimate, contained in Appendix A, on the total variation of the curvature for
continuous, piecewise $W^{2,1}$, closed curves with self--intersections.
This result can be seen as a variant
of the  classical Gauss--Bonnet Theorem for curves.

\section*{Acknowledgments}

 Thanks to a fruitful discussion with Tatsuya Miura
 the proof of Proposition~\ref{minimogoccia} in the current version of the manuscript
 is more detailed than the one 
contained in the published version of this paper.

\section{Notation and preliminary definitions}\label{basic}

In this section we introduce the penalized elastic energy
and the notion of Theta--network.
We also list some useful properties,
whose proofs are based on routine computations 
(we refer to~\cite{daplu} for the details).


When we consider a curve $\gamma$, we mean
a parametrization  $\gamma:[0,1]\to\mathbb{R}^2$.
A curve is of class $C^k$ (or $H^k$) with $k=1,2,\ldots$
if it admits a parametrization $\gamma$ of class $C^k$ (or $H^k$,  respectively).
A curve at least of class $C^1$ 
is said to be regular if $\vert \partial_x \gamma(x) \vert\neq 0$ for every $x\in[0,1]$.

For regular curves we denote by $s$ the arclength parameter and use that 
$\partial_s=\frac{\partial_x}{\vert \partial_x \gamma\vert}$.

\begin{itemize}
\item[-] If a curve $\gamma$ is at least of class $C^1$ and regular we are allowed to speak of its
unit tangent vector
${\tau}=\partial_{s}\gamma=\frac{\partial_x\gamma}{\left|\partial_x \gamma\right|}$.
\item[-] The unit normal vector $\nu$ to $\gamma$ is defined as the 
anticlockwise rotation of $\frac\pi2$ of $\tau$.
\item[-] If a curve $\gamma$ is at least of class $C^2$ and regular we define its scalar curvature as the scalar function $k$ such that $\partial_s^2 \gamma = k\nu$.
\end{itemize}

Notice that  we will adopt the following convention for integrals,
$$
\int_{\gamma}f(\gamma,\tau,\nu,...)\,\mathrm{d}s
=\int_{0}^1f(\gamma,\tau,\nu,...)\vert \partial_x\gamma\vert\,\mathrm{d}x\,,
$$
as the arclength measure is given by $\mathrm{d}s=\vert \partial_x\gamma\vert\mathrm{d}x$
on every curve $\gamma$.

\subsection*{$3$--networks and Theta--networks}

\begin{defn}\label{3network}
A $3$--network is a connected set in the plane, union of three curves 
that  meet at two triple junctions.
A $3$--network is said to be 
\emph{regular} if  every curve
admits a parametrization 
 $\gamma^i$ which is regular, 
\emph{injective} if every curve is injective and \emph{of class $H^2$} if every curve is of class $H^2$.

A \emph{Theta--network} is a $3$--network such that the three curves are
of class $H^2$, and meet at the triple junctions with angles of $120$ degrees.
\end{defn}

Notice that the curves of a Theta--network
can intersect each other or have self-intersections.

Calling $P^1,P^2$ the two triple junctions,
without loss of generality we can parametrized the three curves in such a way that 
$P^1=\gamma^1(0)=\gamma^2(0)=\gamma^3(0)$ and 
$P^2=\gamma^1(1)=\gamma^2(1)=\gamma^3(1)$.
It is also not restrictive to require that $P^1$ coincides with the origin, and 
to fix the unit tangent vectors 
to the curves at $P^1$ as
$\tau^1(0)=(1/2,\sqrt{3}/2)$, $\tau^2(0)=(-1,0)$ and $\tau^3(0)=(1/2,-\sqrt{3}/2)$.
The only information we have at $P^2$
is that $\tau^1(1)+\tau^2(1)+\tau^3(1)=0$.

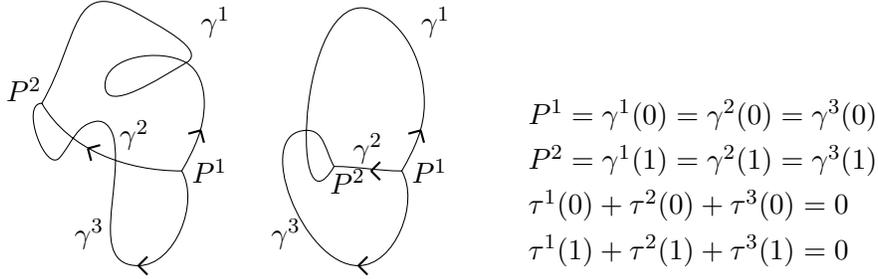
\begin{figure}[H]
\begin{center}
\begin{tikzpicture}[scale=1.2]
\draw[shift={(0,0)}] 
(-1.73,-1.8)
to[out= 180,in=-80, looseness=1] 
(-2,-0.4) 
to[out= 100,in=10, looseness=1] 
(-2.2,-0.2) 
to[out= -170,in=50, looseness=1] 
(-2.6,-0.6) 
to[out= -130,in=180, looseness=1] 
(-2.8,0) 
to[out= 60,in=150, looseness=1.5] (-1.75,1) 
(-2.8,0)
to[out=-60,in=180, looseness=0.9] (-1.25,-0.75)
(-1.75,1)
to[out= -30,in=30, looseness=0.9](-1.2,0.45)
(-2.1,0.2)
to[out= -90,in=-150, looseness=0.9] 
(-1.2,0.45)
(-2.1,0.2)
to[out= 90,in=150, looseness=0.9] 
(-1.2,0.45)
to[out=-30,in=90, looseness=0.9] (-1,0)
to[out= -90,in=60, looseness=0.9] (-1.25,-0.75)
to[out= -60,in=0, looseness=0.9](-1.73,-1.8);
\fill[black](-1.25,-0.75) circle (1pt);
\fill[black] (-2.8,0)  circle (1pt);
\path[shift={(0,0)}] 
(-1.25,-0.75)node[right]{$P^1$}
 (-1.5,-0.35)[left] node{$\gamma^2$}
 (-0.6,.9)[left] node{$\gamma^1$}
 (-2,-1.45)[left] node{$\gamma^3$}
 (-3,0.4) node[below] {$P^2$};
  \draw[thick, scale=1, shift={(-2.29,-0.5)}, rotate=60]
(0,0)to[out= -45,in=135, looseness=1] (0.1,-0.1)
(0,0)to[out= -135,in=45, looseness=1] (-0.1,-0.1);
\draw[ thick, shift={(-1.04,-0.3)}, scale=1, rotate=-30]
(0,0)to[out= -45,in=135, looseness=1] (0.1,-0.1)
(0,0)to[out= -135,in=45, looseness=1] (-0.1,-0.1);
\draw[thick, shift={(-1.73,-1.8)}, scale=1, rotate=90]
(0,0)to[out= -45,in=135, looseness=1] (0.1,-0.1)
(0,0)to[out= -135,in=45, looseness=1] (-0.1,-0.1);
 \end{tikzpicture}
\begin{tikzpicture}[scale=1.2]
\draw[shift={(0,0)}] 
(-2.3,-0.3) 
to[out= 0,in=120, looseness=1] (-2,-0.7) 
to[out= -120,in=150, looseness=1.5] (-1.5,1) 
(-2,-0.7)
to[out=0,in=180, looseness=0.9] (-1.25,-0.75)
(-1.5,1)
to[out= -30,in=90, looseness=0.9] (-1,0)
to[out= -90,in=60, looseness=0.9] (-1.25,-0.75)
to[out= -60,in=0, looseness=0.9](-1.73,-1.8)
(-1.73,-1.8)to[out=180,in=180, looseness=1] (-2.3,-0.3) ;
 \draw[thick, scale=1, shift={(-1.59,-0.75)}, rotate=90]
(0,0)to[out= -45,in=135, looseness=1] (0.1,-0.1)
(0,0)to[out= -135,in=45, looseness=1] (-0.1,-0.1);
\draw[ thick, shift={(-1.04,-0.3)}, scale=1, rotate=-30]
(0,0)to[out= -45,in=135, looseness=1] (0.1,-0.1)
(0,0)to[out= -135,in=45, looseness=1] (-0.1,-0.1);
\draw[thick, shift={(-1.73,-1.8)}, scale=1, rotate=90]
(0,0)to[out= -45,in=135, looseness=1] (0.1,-0.1)
(0,0)to[out= -135,in=45, looseness=1] (-0.1,-0.1);
\fill[black](-1.25,-0.75) circle (1pt);
\fill[black] (-2,-0.7)  circle (1pt);
\path[shift={(0,0)}] 
(-1.25,-0.75)node[right]{$P^1$}
 (-1.35,-0.5)[left] node{$\gamma^2$}
 (-0.6,.9)[left] node{$\gamma^1$}
 (-2.26,-1.45)[left] node{$\gamma^3$}
 (-1.85,-0.62) node[below] {$P^2$};
 \end{tikzpicture}\qquad
\begin{tikzpicture}[scale=1.2]
\path[shift={(0,0)}] 
  (0,0.6)node[right]{$P^1=\gamma^1(0)=\gamma^2(0)=\gamma^3(0)$}
  (0,0.1)node[right]{$P^2=\gamma^1(1)=\gamma^2(1)=\gamma^3(1)$}
 (0,-0.4)node[right]{$\tau^1(0)+\tau^2(0)+\tau^3(0)=0$}
  (0,-0.9)node[right]{$\tau^1(1)+\tau^2(1)+\tau^3(1)=0$};
 \end{tikzpicture}
\end{center}
\caption{Two examples of Theta-networks. Notice that the curves can 
intersect and self-intersect in their interior. The triple junctions, denoted by $P^i$, are only at the end points of the curves.}\label{doppiabolla}
\end{figure}

\subsection*{Penalized elastic energy}

Given a Theta--network $\Gamma$
we let
 $L(\Gamma):=\sum_{i=1}^3\int_{\gamma^i}1\,{\rm{d}}s^i$ be its total length,
  $L(\gamma^i)$ be the length of the $i$-th curve of the network,
 and $E(\Gamma):=\sum_{i=1}^N\int_{\gamma^i}\left(k^i\right)^2\,{\rm{d}}s^i$ 
 be its total elastic energy. Then
 we consider the penalized  elastic energy, defined as:
\begin{equation}\label{Ealpha}
F_\alpha(\Gamma):=
E(\Gamma)+\alpha L(\Gamma), \qquad \alpha>0.
\end{equation}

We are interested in 
the following minimization problem:
\begin{prob}\label{Pdoppiabolla}
Is 
\begin{equation*}
\inf\lbrace F_1\left(\Gamma\right)\vert\; \Gamma\;
\text{is a Theta--network}\rbrace\, \mbox{ attained?}
\end{equation*}
\end{prob}

\subsubsection*{Reduction to the case $\alpha=1$}

It is not restrictive to consider only $F_1$ in Problem~\ref{Pdoppiabolla}.
Indeed, for any Theta--network $\Gamma$, using the scaling properties of the length and of the elastic energy we have
\begin{equation}
F_{1}(\Gamma)
=\alpha^{-\frac12}\, F_{\alpha}\left(\alpha^{-\frac12}\Gamma\right)\, \mbox{ for any }\alpha>0.
\end{equation}
As a consequence if  $\Gamma_1$ 
is a stationary point for $F_{1}$,
then the rescaled network $\Gamma_\alpha:=\alpha^{-\frac12}\,\Gamma_1$ 
is a stationary point for  $F_{\alpha}$, and viceversa.

Hence, from now on we can fix $\alpha=1$
and consider the energy 
\begin{equation*}
F \left(\Gamma\right):=F_1 \left(\Gamma\right)\,.
\end{equation*}
In the following, with a little abuse of notation, we consider the same functional not only on Theta-networks but also on continuous piecewise $H^2$ curves. In all cases $F$ is given by the sums of the elastic energy and of the length for each $H^2$--piece of curve.

\subsubsection*{Optimal rescaling and equipartition of the energy for the minimizers}

Consider a Theta--network $\Gamma$
and the rescaled network $\widetilde{\mathrm{R}}\Gamma$ of $\Gamma$ with the factor 
$\widetilde{\mathrm{R}}:=\sqrt{\frac{E(\Gamma)}{ L(\Gamma)}}$ with $E(\Gamma)$ and $L(\Gamma)$ as defined above \eqref{Ealpha}.
Then for every rescaling of factor $\mathrm{R}>0$ of $\Gamma$ it holds
\begin{equation*}
F(\widetilde{\mathrm{R}}\Gamma)\leq F(\mathrm{R}\Gamma)\,,
\end{equation*}
that is, the network $\widetilde{\mathrm{R}}\Gamma$ is the optimal rescaling (for the energy $F$) of the network $\Gamma$.

We notice that for each optimal rescaled network, and in particular for  the minimizers of $F$ 
(if there exist), there is an equipartition of the energy.
Indeed,
called $\widetilde{\mathrm{R}}$ the optimal rescaling  of a  Theta--network $\Gamma$,
then
\begin{equation}\label{equie}
E(\widetilde{\mathrm{R}}\Gamma)=\sqrt{E(\Gamma)L(\Gamma)}=
L(\widetilde{\mathrm{R}}\Gamma)\,,
\end{equation}
so that 
$$F(\widetilde{\mathrm{R}}\Gamma)= 2 E(\widetilde{\mathrm{R}}\Gamma) = 2 \sqrt{E(\Gamma)L(\Gamma)} \, . $$

\begin{rem} Also in the case of networks, using scaling arguments one can see that minimizing the elastic energy with a penalization of the length is equivalent, up to rescaling, to the minimization problem with
a fixed length constraint on the entire network. For more details see \cite[Sec.5]{daplu}.
\end{rem}

\subsubsection*{Reparametrization with constant velocity}

From now on 
we require that the regular curves 
of a Theta--network are parametrized 
on $[0,1]$
with constant speed equal to the length. 
In this case 
the functional~\eqref{Ealpha} can be written as
\begin{equation}\label{velcost}
F\left(\Gamma\right)
=\sum_{i=1}^3 \frac{1}{L^3(\gamma^i)}\int_{0}^{1}(\gamma_{xx}^i)^2\,{\rm{d}}x
+ L(\gamma^i)\,.
\end{equation}

\subsubsection*{Upper bound on the energy}   
The standard double bubble $\mathcal{B}_{\overline{r}}$
composed by a segment and 
two circular arcs of radius 
$\overline{r}=\sqrt{\frac{8\pi}{3\sqrt{3}+8\pi}}$
is the optimal rescaling of the standard double bubble and hence the one of less energy between all possible standard double bubble. 
Its energy is given by
\begin{equation}\label{energiadellabolla} 
E(\mathcal{B}_{\overline{r}})=\frac{2}{3}\sqrt{8\pi(8\pi+3\sqrt{3})}\approx 18.4059\,.
\end{equation}
Then clearly 
\begin{equation}\label{bddenergy}
\inf\lbrace F\left(\Gamma\right)\vert\; \Gamma\;
\text{is a Theta--network}\rbrace\,\leq \frac{2}{3}\sqrt{8\pi(8\pi+3\sqrt{3})}<+\infty \, .
\end{equation}

          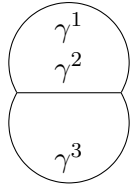
\begin{figure}[H]
\begin{center}
\begin{tikzpicture}[scale=0.8]
\draw[color=black,scale=1,domain=-2.09: 2.09,
smooth,variable=\t,shift={(0,0)},rotate=0]plot({1.*sin(\t r)},
{1.*cos(\t r)}) ; 
\draw[color=black,scale=1,domain=-2.09: 2.09,
smooth,variable=\t,shift={(0,-1)},rotate=180]plot({1.*sin(\t r)},
{1.*cos(\t r)}) ; 
\draw
(-0.87,-0.5)--(0.87,-0.5);
\path[shift={(0,0)}] 
 (0,-0.5)[above] node{$\gamma^2$}
 (0,1)[below] node{$\gamma^1$}
 (0,-2)[above] node{$\gamma^3$};
\end{tikzpicture}
\end{center}    
              \caption{Standard double bubble}\label{bubble}
          \end{figure}

\subsubsection*{Lower bound on the energy} 
\begin{lem}\label{formula2c}
Let $\gamma:[0,1]\to\mathbb{R}^2$ be a regular, continuous, piecewise $H^2$ curve
and suppose that there exists a strictly positive constant $c$ such that 
$\int_{\gamma}\vert k \vert\, {\rm{d}}s\geq c$, then $F(\gamma)\geq 2c$. 
\end{lem}
\begin{proof}
Using H\"{o}lder inequality we get 
$$
c\leq \int_{\gamma}\vert k \vert\, {\rm{d}}s\leq 
\left(\int_\gamma k^2 \, {\rm{d}}s\right)^{\frac12} (L(\gamma))^{\frac12} = (E(\gamma) L(\gamma))^{\frac12}\,,
$$  
then
\begin{equation}
 F(\gamma)=E(\gamma)+L(\gamma)\geq \frac{c^2}{L(\gamma)}+L(\gamma)\geq 2c\,.
\end{equation}
\end{proof}

\begin{lem}
Let $\Gamma$ be a Theta--network. Then $F(\Gamma) \geq 4 \pi \approx 12.5663..$.
\end{lem}
\begin{proof}
Let  $\Gamma=\{\gamma^1,\gamma^2,\gamma^3\}$
be a Theta--network and consider the piecewise $H^2$ closed curves $\Gamma_{i,j}=\{\gamma^i, \gamma^j\}$ with  $i\neq j\in\{1,2,3\}$. 
By Corollary~\ref{variantGB}, we get that 
$
 \int_{\Gamma_{i,j}} \vert k\vert \,{\rm{d}}s\geq \frac{4\pi}{3}\,.
$
and  by the previous Lemma $F(\Gamma_{i,j})\geq\frac{8\pi}{3}$.
Then 
\begin{equation}\label{lowerbound}
F(\Gamma)=\frac12\left(F(\Gamma_{1,2})+F(\Gamma_{2,3})+F(\Gamma_{3,1})\right)\geq 4\pi\,.
\end{equation}
\end{proof}

\section{Existence of minimizers}\label{esistenza}

As explained in the Introduction we introduce now the functional $\overline{F}$, 
an extension of $F$ to all $3$-networks of class $H^2$.
We prove existence of a minimizer for $\overline{F}$ in this larger space and show that  $\overline{F}$ is the relaxation of $F$ with respect to $H^2$ weak convergence.

\begin{defn}\label{classNandM}
A  \textit{``degenerate'' 
Theta--network} is a network  composed by two regular curves $\gamma^1,\gamma^2$
of class $H^2$,
forming angles in pairs of $120$ and $60$
degrees
and by a curve $\gamma^3$ of length zero.
\end{defn}

\begin{figure}[H]
\begin{center}
\begin{tikzpicture}[scale=2]
\draw[white]
(0,-1.1)--(1,-1.1);
\draw[thick]
(0,0)to[out= 60,in=0, looseness=1] (0,0.75)
(0,0.75)to[out= 180,in=120, looseness=1] (0,0)
(0,0)to[out= -60,in=90, looseness=1] (0.3,-0.55)
(0.3,-0.55)to[out= -90,in=0, looseness=1] (0,-0.9)
(0,0.75)to[out= 180,in=120, looseness=1] (0,0)
(0,0)to[out= -120,in=90, looseness=1] (-0.4,-0.65)
(-0.4,-0.65)to[out= -90,in=180, looseness=1] (0,-0.9);
\fill[black](0,0) circle (0.7pt);
\path[shift={(0,0)}] 
(0,0)[right] node{$P$}
 (0.13,0.45)[right] node{$\gamma^1$}
 (0.3,-0.6)[right] node{$\gamma^2$};
\end{tikzpicture}\qquad\quad
\begin{tikzpicture}[scale=2]
\draw[thick]
(0,0)to[out= 60,in=0, looseness=1] (0,0.75)
(0,0.75)to[out= 180,in=120, looseness=1] (0,0)
(0,0)to[out= -60,in=-90, looseness=1] (0.7,0.55)
(0.7,0.55)to[out= 90,in=0, looseness=1] (0,1.5)
(0,0)to[out= -120,in=-90, looseness=1] (-0.8,0.65)
(-0.8,0.65)to[out= 90,in=180, looseness=1] (0,1.5);
\fill[black](0,0) circle (0.7pt);
\path[shift={(0,0)}] 
(0,-0.07)[below] node{$P$}
 (0.13,0.45)[right] node{$\gamma^1$}
 (0.6,1)[right] node{$\gamma^2$};
\end{tikzpicture}\qquad\qquad\;
\begin{tikzpicture}[scale=2]
\draw[thick, color=black,scale=0.3,domain=-3.15: 3.15,
smooth,variable=\t,rotate=30, shift={(-1,0)}]plot({1.*sin(\t r)},
{2.*cos(\t r)}) ; 
\draw[thick, color=black,scale=0.3,domain=-3.15: 3.15,
smooth,variable=\t,rotate=-30, shift={(-0.6,0.4)}]plot({0.6*sin(\t r)},
{2.2*cos(\t r)}) ; 
\fill[black](0,0) circle (0.7pt);
\path[shift={(0,0)}] 
(0,-0.07)[right] node{$P$}
 (0.23,0.45)[right] node{$\gamma^1$}
 (-0.63,-0.65)[right] node{$\gamma^2$};
 \draw[white]
(0,-0.9)--(1,-0.9);
\end{tikzpicture}
\end{center}    
              \caption{Three examples of ``degenerate'' Theta--network.
              Notice that the curves can 
intersect and self-intersect in their interior. The four points, denoted by $P$,
are only at the end points of the curves.}\label{exinM}
          \end{figure}
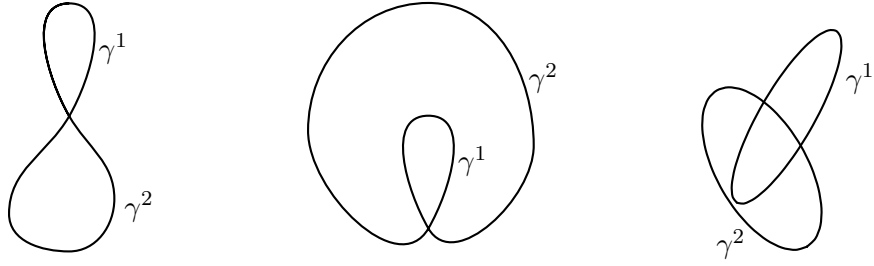

From now on we parametrized all the regular curves of a $3$--network on $[0,1]$
and with constant speed equal to their length. 
In particular,
the energy $F$ on regular $3$--networks is given as in~\eqref{velcost}.

\begin{defn}\label{funzionalerilassato}
We define the  functional  $\overline{F}$ 
on $3$--networks $\Gamma$ of class $H^2$ 
as follows
$$
\overline{F}(\Gamma)=
\begin{cases}
\sum_{i=1}^3
\frac{1}{L^3(\gamma^i)}
\int_{0}^1 (\gamma_{xx}^i)^{2}\, {\rm{d}}x+ L(\gamma^i) \qquad \text{if}\; \Gamma
\;\text{is Theta--network, }\\
\sum_{i=1}^2
\frac{1}{L^3(\gamma^i)}
\int_{0}^1 (\gamma_{xx}^i)^{2}\, {\rm{d}}x+ L(\gamma^i)  \qquad \text{if}\; \Gamma
\;\text{is a ``degenerate'' 
Theta network, }\\
+\infty \qquad\qquad\qquad\;\;\,\quad\qquad\quad\quad\quad\;\;\,\,\,\, \text{otherwise}
\end{cases}
$$
\end{defn}

By definition the two functionals $F$ and $\overline{F}$
coincide on Theta--networks.
Notice that when we speak of the set of all 
$3$--networks of class $H^2$ we
do not ask that the lengths of the three curves constituting the network are positive. 
In particular, we do not require regularity of the curves.

\begin{thm}\label{rilassato}
The functional $\overline{F}$ is the relaxation of the functional $F$
in $H^2$,
that is for any $3$-network $\Gamma$ of class $H^2$
$$
\overline{F}(\Gamma)=\inf\{\liminf_{n\to\infty}{F}(\Gamma_n)
:\Gamma_n\rightharpoonup\Gamma\,\text{weakly in}\, H^2,\;\text{with}\;\Gamma_n\;\text{Theta--network}\,\}\,.
$$

\end{thm}

Before proving Theorem \ref{rilassato} we discuss the compactness properties and the semicontinuity of the 
functional $\overline{F}$.
 
\subsubsection*{Compactness}

As for the Theta--networks, without loss of generality 
we can ask that 
one of the  two triple junctions of a $3$--network coincides with the origin.

\begin{prop}\label{compattezza}
Let $\Gamma_n=\{\gamma^1_n,\gamma^2_n,\gamma^3_n\}$
be a sequence of
$3$--networks of class $H^2$
such that
$$
\limsup_n \overline{F}(\Gamma_n)<+\infty\,.
$$
Then $\Gamma_n$ converge, up to a subsequence, 
weakly in $H^2$ and strongly
in $C^{1,\alpha}$ with $\alpha\in(0,\frac12 )$ to
$\Gamma$.
The network $\Gamma$ is either a Theta--network or a ``degenerate'' 
Theta--network.
\end{prop}
\begin{proof}
Up to a subsequence (not relabeled) all $\Gamma_n$ are 
either Theta--networks or  ``degenerate'' 
Theta--networks. 
\begin{enumerate}
\item  Suppose they are
all Theta--networks, 
i.e. $\Gamma_n=\{\gamma_n^1,\gamma_n^2,\gamma_n^3\}$. For $x\in\{0,1\}$ we have
$\gamma^1_n(x)=\gamma^2_n(x)=\gamma^3_n(x)$ 
and $\sum_{i=1}^3\tau_n^i(x)=0$.

From the bound on the energy it follows that 
the length of each curve of the network $\Gamma_n$ is uniformly bounded from above.
As a consequence, for $i=1,2,3$
$$\int_0^1(\gamma^i_{n,xx})^2\,\mathrm{d}x
\leq \int_0^1 \frac{(\gamma^i_{n,xx})^2}{L(\gamma^i_n)^{3}}\,\mathrm{d}x L(\gamma^i_n)^{3}
\leq \left(\sum_{i=1}^3\int_0^1 \frac{(\gamma^i_{n,xx})^2}{L(\gamma^i_n)^{3}}\,\mathrm{d}x \right) L(\gamma^i_n)^{3} \leq F(\Gamma_n)L(\gamma^i_n)^{3} \, ,$$
from which one derives a uniform bound on the $L^2$--norm of the second derivatives. 
As each network contains the origin and the curves are parametrized with constant speed 
(equal to the length)
we have 
$\Vert \gamma_n^i\Vert_\infty\leq C$ and 
$\Vert{\gamma^{i}}_{n,x} \Vert_\infty\leq C'$
with $C,C'$ positive constants. Then, up to a subsequence (not relabeled), $\gamma^i_{n}\rightharpoonup \gamma^i_\infty$ 
weakly in $H^2(0,1)$ and 
$\gamma^i_{n}\rightarrow \gamma^i_\infty$
strongly in $C^{1,\alpha}([0,1])$ for every $\alpha\in(0,\frac12)$. It remains to understand what the limit is.

\smallskip

We claim that for at least two of the curves we have uniform bounds from below on the length. Indeed, suppose by contradiction that the length of at  least two curves of
 $\Gamma_n$ tends to zero. 
Let us say that $L(\gamma^2_n)\to 0$ and $L(\gamma^3_n)\to 0$.
Calling $\Gamma_{n,2,3}=\{\gamma_n^2,\gamma_n^3\}$,
by Corollary~\ref{variantGB} we have
\begin{equation}\label{eq:boundbelow}
\overline{F}(\Gamma_n)\geq \int_{\Gamma_{n,2,3}} k^2\,{\rm{d}}s
\geq\frac{16\pi^2}{9(L(\gamma_n^2)+L(\gamma_n^3))}\to +\infty
\end{equation}
a contradiction. 
Hence on at least two of the sequences $\gamma_n^i$, $i\in\{1,2,3\}$, the length is uniformly bounded from below. 

\begin{itemize}
\item[-] Suppose that none of the lengths of the curves in $\Gamma_n$ goes to zero. Then
due to the strong convergence in $C^1$ of
 $\gamma_n^i\to \gamma_{\infty}^i$ for all $i\in\{1,2,3\}$, 
all the tangents $\tau^i_\infty$
are well defined and also the conditions at the junctions $\sum_{i=1}^3\tau^i_\infty(x)=0$ and $\gamma^1_{\infty}(x)=\gamma^2_{\infty}(x)=\gamma^3_{\infty}(x)$, $x\in\{0,1\}$, are fulfilled. 
Hence $\Gamma$ is a Theta--network.
\item[-] Suppose now that the length of one curve of 
$\Gamma_n$ tends to zero. Let us fix $L(\gamma^3_n)\to0$.
We have
\begin{align*}
\vert \tau^3_n(1)-\tau^3_n(0)\vert
& =\frac{1}{L(\gamma^3_n)} \left| \int_0^1 \gamma^3_{n,xx}(x)\,\mathrm{d}x\, \right| \\
& \leq  \frac{1}{L(\gamma^3_n)} \left( \int_0^1 \left(  \gamma^3_{n,xx}(x)\right) ^2\,\mathrm{d}x\right)^{1/2}  \left( \int_0^1 1\,\mathrm{d}x\right)^{1/2}
 \\
 &
= \left( \int_0^1  \frac{(\gamma^3_{n,xx}(x))^2}
 {(L(\gamma^3_n))^3} \,\mathrm{d}x\right)^{1/2}(L(\gamma^3_n))^{1/2}\\
& \leq (\overline{F}(\Gamma_n))^{1/2} (L(\gamma^3_n))^{1/2}
\leq C (L(\gamma^3_n))^{1/2}\,.
\end{align*}
Hence $\vert \tau^3_n(1)-\tau^3_n(0)\vert\to 0$ as $n\to\infty$,
that combined with the $C^1$
convergence gives that  $\Gamma$ is  a ``degenerate'' 
Theta--network.
\end{itemize}
\item Suppose instead that all the terms of $\Gamma_n$ are ``degenerate'' 
Theta--networks, i.e. $\Gamma_n=\{\gamma_n^1,\gamma_n^2\}$. 
In this case, none of the length goes to zero. Indeed, suppose by contradiction that $L(\gamma^i_n)\to 0$ as $n\to\infty$, for some $i \in \{1,2\}$,
using Corollary~\ref{variantGB1curve} we know that
\begin{equation*}
\pi \leq \left(\int_{\gamma^i_n} k^2 \,{\rm{d}}s \right)^{\frac12} L(\gamma^i_n)^{\frac12},
\end{equation*}
and hence
\begin{equation*}
\overline{F}(\Gamma_n)\geq \int_{\gamma^i_n} k^2\,{\rm{d}}s
>\frac{\pi^2}{L(\gamma^i_n)}\to +\infty\,,
\end{equation*}
a contradiction. 
Due to the upper bound on the length, repeating the same arguments as before
we have again that $\gamma_n^i$ converges weakly in $H^2$ and strongly in $C^1$ to $\gamma_{\infty}^i$, for $i=1,2$. As no length goes to zero and there is convergence in $C^1$, the angles between the curves are preserved
passing to the limit. As a consequence, $\Gamma=\{\gamma_{\infty}^1,\gamma_{\infty}^2\}$ is a ``degenerate''  Theta--network.
\end{enumerate}
\end{proof}

\subsection*{Semicontinuity}

\begin{prop}
The  functional $\overline{F}$ on $3$-networks of class $H^2$ 
is weakly lower semicontinuous in $H^2$.
\end{prop}

\begin{proof}
Let $\Gamma_n$ be a sequence of $3$-networks converging weakly in $H^2$ to $\Gamma_{\infty}$.
Passing to a subsequence (not relabeled) we can reduce to prove that
$\lim_{n}\overline{F}(\Gamma_{n})\geq \overline{F}(\Gamma_\infty)$. If $\lim_{n}\overline{F}(\Gamma_{n})=\infty$ there is nothing to prove,
and hence we can suppose that the limit of $\overline{F}(\Gamma_{n})$ is finite.
As a consequence, we can assume that all the terms of $\Gamma_n$ are either
Theta--networks or  ``degenerate'' 
Theta--networks.\smallskip

For a single regular $H^2$ curve $\gamma$
(of positive length) the functional 
$\frac{1}{L^3(\gamma)}\int_0^1\gamma^2_{xx}\,{\rm{d}}x+L(\gamma)$ 
is well defined and also weakly 
lower semicontinuous on this class. 
This follows from the continuity of the length  and and the 
lower semicontinuity of the $L^2$ norm.
In the case in which $\Gamma_n$ and the limit $\Gamma$ are all
Theta--networks 
(or all ``degenerate''  Theta--networks), 
the claim follows using
for all the three curves (or the two curves) of the network the previous argument for a single curve.
It remains to consider the case in which $\Gamma_n$ are Theta--networks 
and $\Gamma$ is a ``degenerate''  Theta--network. We may assume  that 
$L(\gamma^3)=0$ and the claim follows
also in this case with the same arguments
noticing that 
$$\sum_{i=1}^3 \frac{1}{L^3(\gamma^i_n)}
 \int_0^1((\gamma_{xx})_n^i)^2\,{\rm{d}}x+ L(\gamma^i_n)
\geq \sum_{i=1}^2 \frac{1}{L^3(\gamma^i_n)} \int_0^1 ((\gamma_{xx})_n^i)^2\,{\rm{d}}x+ L(\gamma^i_n) \, .$$
and then we can reduce to the case of ``degenerate''  Theta--networks.
\end{proof}

\subsection*{Existence}

Combining the compactness and the lower semicontinuity results we obtain the following:

\begin{cor}[Existence of minimizer for $\overline{F}$]\label{esistenzatrenet}
There exists $\Gamma_{\min}$ Theta--network or ``degenerate''  Theta--network  
minimizer of the functional $\overline{F}$ among $3$--networks of class $H^2$.
\end{cor}

We can finally prove that the functional $\bar{F}$ is the relaxation of $F$ on Theta--networks in $H^2$.

\begin{proof}[{Proof of Theorem~\ref{rilassato}}]
The claim follows once we have proved that
for every ``degenerate''  Theta--network $\Gamma$  
there exists a sequence $\Gamma_n$ of Theta--networks such that 
$$
\Gamma_n\rightharpoonup\Gamma \mbox{ in }H^2 \mbox{ and }F(\Gamma_n)\to \overline{F}(\Gamma) \, . 
$$
Let us denote by $\gamma_1,\gamma_2$ the two curves 
constituting the ``degenerate''  Theta--network $\Gamma$. Without loss of generality we
put the four--point $P:=\gamma^1(0)=\gamma^1(1)=\gamma^2(0)=\gamma^2(1)$ at the origin
and orient the network in such a way that $\tau^1(0)$
forms an angle of $60$ degrees with the $x$--axis
(in particular the second component of the vector is positive).
As the angles at $P$ are fixed,  $\tau^1(1)$
forms an angle of $-60$ degrees with the $x$--axis
(in particular the second component of the vector is negative).
Then by the intermediate value Theorem there exists at least one point  $t_1\in[0,1]$ 
such that $\tau^1(t_1)$ 
is horizontal (the second component of
the vector is zero). 
Repeating the same argument for the curve $\gamma^2$ one finds also a $t_2\in[0,1]$ 
such that $\tau^2(t_2)$ 
is horizontal. Call $Q^1=\gamma^1(t_1)$ and $Q^2=\gamma^2(t_2)$.
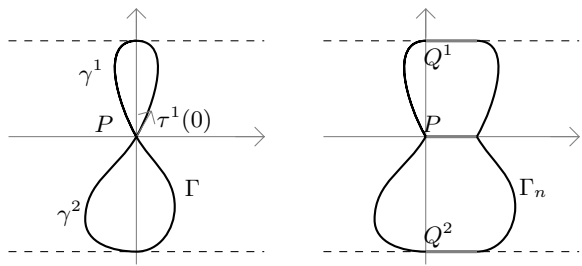
\begin{figure}[H]
\begin{center}
\begin{tikzpicture}[scale=1.7]
\draw[black!50!white, shift={(1,0)}, scale=0.75, rotate=-90]
(0,0)to[out= -45,in=135, looseness=1] (0.1,-0.1)
(0,0)to[out= -135,in=45, looseness=1] (-0.1,-0.1)
(0,-0.3)to[out= 90,in=-90, looseness=1] (0,0.003);
\draw[black!50!white, shift={(0,1)}, scale=0.75, rotate=0]
(0,0)to[out= -45,in=135, looseness=1] (0.1,-0.1)
(0,0)to[out= -135,in=45, looseness=1] (-0.1,-0.1)
(0,-0.3)to[out= 90,in=-90, looseness=1] (0,0.003);
\draw[black!50!white]
(-1,0)--(1,0)
(0,-1)--(0,1);
\draw[dashed]
(-1,0.75)--(1,0.75)
(-1,-0.9)--(1,-0.9);
  \draw[thick]
(0,0)to[out= 60,in=0, looseness=1] (0.2,0.75)
(0.2,0.75)to[out= 180,in=120, looseness=1] (0,0)
(0,0)to[out= -60,in=90, looseness=1] (0.3,-0.55)
(0.3,-0.55)to[out= -90,in=0, looseness=1] (-0.15,-0.9)
(0,0)to[out= -120,in=90, looseness=1] (-0.4,-0.65)
(-0.4,-0.65)to[out= -90,in=180, looseness=1] (-0.15,-0.9);
\draw[white]
(0,-1.1)--(1,-1.1);
\draw[black!50!white, shift={(0.117,0.2)}, scale=0.75, rotate=-30]
(0,0)to[out= -45,in=135, looseness=1] (0.1,-0.1)
(0,0)to[out= -135,in=45, looseness=1] (-0.1,-0.1)
(0,-0.3)to[out= 90,in=-90, looseness=1] (0,0.003);
\path[font=\footnotesize, shift={(0,0)}] 
 (0.3,-0.4)[right] node{$\Gamma$}
 (-0.1,0.1)[left] node{$P$}
 (0.08,0.13)[right] node{$\tau^1(0)$}
 (-0.53,0.5)[right] node{$\gamma^1$}
 (-0.7,-0.6)[right] node{$\gamma^2$};
\end{tikzpicture}\qquad
\begin{tikzpicture}[scale=1.7]
\draw[black!50!white, shift={(1,0)}, scale=0.75, rotate=-90]
(0,0)to[out= -45,in=135, looseness=1] (0.1,-0.1)
(0,0)to[out= -135,in=45, looseness=1] (-0.1,-0.1)
(0,-0.3)to[out= 90,in=-90, looseness=1] (0,0.003);
\draw[black!50!white, shift={(-0.2,1)}, scale=0.75, rotate=0]
(0,0)to[out= -45,in=135, looseness=1] (0.1,-0.1)
(0,0)to[out= -135,in=45, looseness=1] (-0.1,-0.1)
(0,-0.3)to[out= 90,in=-90, looseness=1] (0,0.003);
\draw[black!50!white]
(-1,0)--(1,0)
(-0.2,-1)--(-0.2,1);
\draw[dashed]
(-1,0.75)--(1,0.75)
(-1,-0.9)--(1,-0.9);
\draw[gray, very thick]
(-0.2,0)--(0.2,0)
(0,0.75)--(0.4,0.75)
(-0.35,-0.9)--(0.05,-0.9);
  \draw[thick]
(0.2,0)to[out= 60,in=0, looseness=1] (0.4,0.75)
(0,0.75)to[out= 180,in=120, looseness=1] (-0.2,0)
(0.2,0)to[out= -60,in=90, looseness=1] (0.5,-0.55)
(0.5,-0.55)to[out= -90,in=0, looseness=1] (0.05,-0.9)
(-0.2,0)to[out= -120,in=90, looseness=1] (-0.6,-0.65)
(-0.6,-0.65)to[out= -90,in=180, looseness=1] (-0.35,-0.9);
\draw[white]
(0,-1.1)--(1,-1.1);
\path[font=\footnotesize, shift={(0,0)}] 
 (0.45,-0.4)[right] node{$\Gamma_n$}
 (-0.1,0.8)[below] node{$Q^1$}
 (-0.1,-0.95)[above] node{$Q^2$}
 (0,0.1)[left] node{$P$};
\end{tikzpicture}
\end{center}
\caption{Construction of the recovery sequence.}\label{figrecovery}
\end{figure}

We construct the (recovery) sequence of Theta--networks $\Gamma_n$ cutting $\Gamma$ at the 
points $P,Q^1$ and $Q^2$ and gluing three horizontal segments of length $1/n$ (see Figure \ref{figrecovery}). 
Then $\Gamma_n$ is indeed a Theta--network for each $n$ and 
it is easy to see that $\Gamma_n\rightharpoonup \Gamma$ weakly in $H^2$.  Moreover 
$F(\Gamma_n)\to \overline{F}(\Gamma)$ as the 
pieces of $F(\Gamma_n)$ and $\overline{F}(\Gamma)$ with non zero
curvature are the same
and the total length of $\Gamma_n$ converge to the length of $\Gamma$ as $n\to\infty$.
\end{proof}

\section{Properties of minimizers}\label{prop}

In this section we want to prove some properties of the minimizers for the relaxed functional
$\overline{F}$.
In particular we show that the minimizers are actually Theta--networks 
and hence we give a positive
answer to
Problem~\ref{Pdoppiabolla}.
We get our result
showing that the minimal energy of a 
``degenerate" Theta--network is higher than the energy
of a given competitor among Theta--networks (the standard double bubble).
A tool to obtain the desired lower bound is studying the
minimization problem  for drops and double drops.

\medskip

Before studying these other minimization problems we shortly discuss the Euler-Lagrange equations and the regularity of 
the curves constituting a $\overline{F}$-minimizing network.

\begin{prop}
Let $\widetilde{\Gamma}$ be a minimizer for $\overline{F}$ and $\widetilde{\gamma}$ be 
any of the regular curves constituting $\widetilde{\Gamma}$. Then
\begin{enumerate}
\item the curve 
$\widetilde{\gamma}$ minimizes the penalized elastic energy $\int_\gamma k^2+1\,{\rm{d}}s$ among regular $H^2$ curves $\gamma:[0,1]\to \mathbb{R}^2$ satisfying 
the same boundary conditions as $\widetilde{\gamma}$, i.e. $\gamma(x)=\widetilde{\gamma}(x)$ and $\partial_s \gamma(x)= \partial_s \widetilde{\gamma}(x)$ at $x\in \{0,1\}$.
\item  the curve 
$\widetilde{\gamma}$ is $C^{\infty}([0,1])$ and solves the equation
\begin{equation}\label{EL:curve}
2 \partial_s^2 k +k^3-k =0 \mbox{ on }(0,1)\, . 
\end{equation}
\end{enumerate} 
\end{prop}
 \begin{proof}
\begin{enumerate}
\item Let $\Gamma$ be the regular network that we obtain from the network $\widetilde{\Gamma}$ replacing the curve $\widetilde{\gamma}$ with a regular $H^2$ curve $\gamma$ that satisfies the 
same boundary conditions as $\widetilde{\gamma}$, as explained in the statement. Then $\Gamma$ is a Theta--network if $\widetilde{\Gamma}$ is a Theta--network, whereas $\Gamma$ is a ``degenerate" Theta--network if so is $\widetilde{\Gamma}$. 
Since by assumption $\overline{F}(\widetilde{\Gamma}) \leq \overline{F}(\Gamma)$ and we did not change the other curves (or the other curve) constituting the network, it follows that $\widetilde{\gamma}$ minimizes the penalized elastic energy among curves satisfying the same boundary conditions.
\item For smooth curves the Euler-Lagrange equation (\ref{EL:curve}) is computed in \cite{langsing}. In this weak setting, as usual one starts from the first variation that corresponds to the weak formulation of (\ref{EL:curve}). Then by appropriate choice of the test functions in $H^2_0(0,1)$ one show first that $k \in L^{\infty}$, then $k \in W^{1,\infty}$ and finally that $\tilde{\gamma} \in H^4$. By elliptic regularity and a bootstrap argument one gets that $\widetilde{\gamma}$ is in $C^{\infty}([0,1])$ and in particular that (\ref{EL:curve}) is satisfied pointwise. These reasonings have been carried out in details for instance in \cite[Prop.4.1]{daplu}.
\end{enumerate}
\end{proof}

\subsection{Drops}\label{gocce}

\begin{defn}\label{def:drop}
We call \emph{drop} a regular curve $\gamma:[0,1]\to\mathbb{R}^2$ of class $H^2$ such that $\gamma(0)=\gamma(1)$.
\end{defn}
Notice that in the definition of drops
we do not impose any assumption on the tangents
at $x=0$ and at $x=1$.
In the following every time we consider a drop, without loss of generality,  
the point $\gamma(0)=\gamma(1)$ is fixed at the origin of the axis.

\begin{figure}[H]
\begin{center}
\begin{tikzpicture}[scale=2]
 \draw[black!60!white, shift={(2,-1)}, scale=0.75, rotate=-90]
(0,0)to[out= -45,in=135, looseness=1] (0.1,-0.1)
(0,0)to[out= -135,in=45, looseness=1] (-0.1,-0.1)
(0,-0.3)to[out= 90,in=-90, looseness=1] (0,0.003);
\draw[black!60!white, shift={(0,0.5)}, scale=0.75, rotate=0]
(0,0)to[out= -45,in=135, looseness=1] (0.1,-0.1)
(0,0)to[out= -135,in=45, looseness=1] (-0.1,-0.1)
(0,-0.3)to[out= 90,in=-90, looseness=1] (0,0.003);
\draw[black!30!white]
(0,-1.25)--(0,0.5)
(-2,-1)--(0,-1)
(2,-1)--(0,-1);
\draw[thick, shift={(0,-1)}]
(0,0)to[out= 20,in=-90, looseness=1] (0.7,0.55)
(0,0)to[out= 160,in=-90, looseness=1] (-0.9,0.65)
(-0.9,0.65)to[out=90,in=-90, looseness=1] (-1.3,0.9)
(-1.3,0.9)to[out=90,in=-160, looseness=1](-0.8,1.2)
(-0.8,1.2)to[out=20,in=180, looseness=1] (-0.2,0.9)
(0.7,0.55)to[out= 90,in=180, looseness=1] (1,0.9)
(1,0.9)to[out= 0,in=0, looseness=1] (0.7,0.55)
(0.7,0.55)to[out= 180,in=-30, looseness=1] (0.3,0.75)
(0.3,0.75)to[out=150,in=0, looseness=1] (-0.2,0.9);
\path[shift={(0,0)}] 
 (0.13,0)[right] node{$\gamma$};
\end{tikzpicture}

\end{center}    
              \caption{An example of a drop.}\label{drop}
          \end{figure}
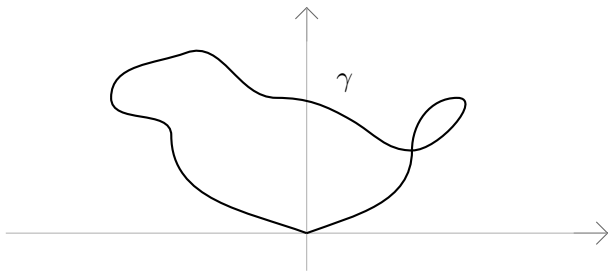

We consider now the restriction of our minimization problem to the class of drops.  
Let us recall that $F$ denotes also 
the penalized elastic energy functional on regular $H^2$ curves  
$\gamma:[0,1]\to\mathbb{R}^2$, that is 
$F(\gamma)=\int_\gamma k^2+1\,{\rm{d}}s$.

\begin{prob}\label{problemdrops}
Is $$\inf\{F(\gamma)\vert \;\gamma\, \text{ is a drop}\} \mbox{ attained?} $$
\end{prob}

We observe that Problem~\ref{problemdrops} is slightly different from Problem~\ref{Pdoppiabolla}, 
since no condition is imposed on the angle  at $\gamma(0)=\gamma(1)$.
Since the circle of radius $1$ is a particular drop, we immediately get
\begin{equation}\label{limitato}
m:=\inf\{F(\gamma)\vert\; \gamma\, \text{is a drop}\,\}\leq 4\pi<\infty\,.
\end{equation}

\begin{thm}[Existence of minimizers]\label{esistenzaminimogoccia}
There exists a drop $\gamma:[0,1]\to\mathbb{R}^2$ such that $F(\gamma)=m$. 
That is, a minimizer for Problem \ref{problemdrops} exists.
\end{thm}

\begin{proof}
We consider a minimizing sequence of drops $\gamma_n:[0,1]\to\mathbb{R}^2$. Since $L(\gamma)\leq F(\gamma)$, the length of the curves $L(\gamma_n)$ is uniformly bounded from above.
Moreover 
combining Corollary~\ref{variantGB1curve}
and H\"{o}lder inequality we obtain
\begin{equation*}
\pi \leq  \int_{\gamma_n} \vert k_n \vert \,{\rm{d}}s_n\leq
\left( \int_{\gamma_n} k_n^2\,{\rm{d}}s_n\right)^{1/2} (L(\gamma_n))^{1/2}\leq 
(F(\gamma_n))^{1/2} (L(\gamma_n))^{1/2}\,.
\end{equation*}
As $F(\gamma_n)\leq C<\infty$, for all $n \in \mathbb{N}$ we get the uniform bound
\begin{equation}\label{lunghezza}
0<\frac{\pi^2}{C}\leq L(\gamma_n)\leq C\,. 
\end{equation}

By parametrizing 
the sequence with constant velocity equal to the length we obtain
$\Vert \gamma_{n,x}\Vert_\infty=L(\gamma_n)\leq C$.
As we have fixed the point $\gamma_n(0)=\gamma_n(1)$ at the origin, 
we also get $\Vert \gamma_n\Vert_\infty\leq C$.
By~\eqref{velcost} we obtain 
$ \frac{1}{L^3(\gamma_n)}
\int_0^1 \gamma^2_{n,xx}\,{\rm{d}}x\leq F(\gamma_n)\leq C$, thus 
we have
$$
\Vert \gamma_{n,xx}\Vert _2\leq L^3(\gamma_n) C\leq C^4\,.
$$
Then there exist a subsequence $\gamma_{n_{k}}$ and $\gamma\in H^2$ such that
$\gamma_{n_{k}}$ converge weakly in $H^2$ and strongly in 
$C^{1,\alpha}$ (with $\alpha \in (0,\frac12)$) to $\gamma$. 
Because of \eqref{lunghezza} the limit has 
strictly positive length.
The functional $F$ is weakly lower semicontinuous, hence
$F(\gamma)\leq \lim_k F(\gamma_{n_k})=m$, and this concludes the proof.
\end{proof}

In Subsection~\ref{minimum} we give explicitly the unique 
(up to isometries of $\mathbb{R}^2$) minimizer
for Problem~\ref{problemdrops}.

\subsection{Double drops}
\begin{defn}\label{dd}
We call 
\begin{enumerate}
\item \emph{double drop}
a $2$--network $\mathcal{D}=\{\gamma^1,\gamma^2\}$,
such that for $i\in\{1,2\}$ the curves $\gamma^i:[0,1]\to\mathbb{R}^2$  are drops and there exists
a four--point $\gamma^1(0)=\gamma^1(1)=\gamma^2(0)=\gamma^2(1)=:P$;
\item \emph{symmetric double drop} $\mathcal{SD}$
a double drop in which the 
two drops are symmetric with respect to 
the four--point $P$.
\end{enumerate}
\end{defn}

Notice that a symmetric double drop can be seen as a double drop $\mathcal{D}=\{\gamma_1,\gamma_2\}$ where $\gamma_2(t)=-\gamma_1(1-t)$ for all $t \in [0,1]$. 
The penalized elastic energy for a network of two curves is simply 
the sum of the penalized elastic energy of
the curves and for a symmetric double drop $\mathcal{SD}=\{\gamma^1,\gamma^2\}$ 
we have that $F(\mathcal{SD})=2F(\gamma^1)=2F(\gamma^2)$.

Again, as for the drops, for simplicity we fix the four--point $P$ at the origin of the axes.
From the definition of symmetric double drop we have 
that the curves $\gamma^1$ and $\gamma^2$ form at $P$ angles equal in pairs.

\begin{rem}\label{globallyreg}
Due to symmetry any symmetric double drop is globally of class $H^2$.
More precisely: given 
a symmetric double drop  $\mathcal{SD}:[0,1]\to\mathbb{R}^2$, 
 there exists a reparametrization of  $\mathcal{SD}$ such that 
 $\mathcal{SD}$ is a closed curve
globally of class $H^2$.
\end{rem}

\subsection{Characterization of optimal drops}\label{minimum}
\begin{lem}\label{allmin}
Let us call
\begin{align*}
n&:=\inf\{F(\mathcal{D})\vert\; \mathcal{D}\, \text{is a  double drop}\,\}\\
\tilde{n}&:=\inf\{F(\mathcal{SD})\vert\; \mathcal{SD}\, \text{is a symmetric double drop}\,\}\\
\overline{n}&:=\inf\{F(\varphi)\vert \; \varphi:\mathbb{S}^1\to\mathbb{R}^2,\, \varphi\, 
\text{is a regular curve of class }\,H^2\,  \nonumber \\ 
& \qquad \quad  \text{and there exists at least two points}\, 
s_1\neq s_2\in \mathbb{S}^1\,\text{such that}\;\gamma(s_1)=\gamma(s_2)\}\,.
\end{align*}
Then 
$$
n=\tilde{n}=\overline{n}=2m
$$
with $m$ as defined in~\eqref{limitato}.
Moreover all the infima are minima.
\end{lem}
\begin{proof}
By Definition \ref{dd} and Remark \ref{globallyreg} one sees that $n \leq \overline{n} \leq \tilde{n}$. Moreover in these three
minimization problems a candidate for the minimum is the symmetric double drop $\mathcal{D}_{\min}=\{\gamma^1,\gamma^2\}$ with $\gamma^1$ a minimizer 
for Problem \ref{problemdrops} (whose existence is established in Theorem~\ref{esistenzaminimogoccia}) and $\gamma^2$ the curve such that
$\gamma^2(t)=-\gamma^1(1-t)$, $t \in [0,1]$. Since $F(\mathcal{D}_{\min})=2m$, with $m$ defined in \eqref{limitato}, it follows that
$$ n \leq \overline{n} \leq \tilde{n} \leq 2m \, . $$
On the other hand,
a double drop $\mathcal{D}$ is a union 
of two different drops
 $\varphi^i:[0,1]\to\mathbb{R}^2$ of class $H^2$. We can suppose without loss of generality that $F(\varphi^2)\geq F(\varphi^1)$, hence
 $$
 F(\mathcal{D})=F(\varphi^1)+F(\varphi^2)\geq 2F(\varphi^1)\geq 2m\,.
 $$
We conclude that $
 2m=\tilde{n}\geq \overline{n}\geq n\geq 2m$
and, that in particular, the three infima are actually equal.
Since $F(\mathcal{D}_{\min})=2m$ and $\mathcal{D}_{\min}$ belongs to all three sets, it is a minimizer
for all three problems.
\end{proof}

We are now ready to characterize the optimal drop of Theorem~\ref{esistenzaminimogoccia}. Let us recall that Langer and Singer established in \cite[Theorem 0.1(a)]{langsing1} that the  ``Figure Eight" is the unique closed planar elastica (up to multiple coverings and isometries of $\mathbb{R}^2$) with self--intersections. Here the word elastica refers to a critical point of the functional $F$ on closed curves.

\begin{prop}\label{minimogoccia}
Up to isometries of $\mathbb{R}^2$ 
the ``Figure Eight" (see Figure \ref{otto}) is the unique minimizer of 
the functional $F$ between all planar closed curves with at least one self-intersection. One of the two drops of $\mathcal{F}$ is the 
unique minimizer for Problem~\ref{problemdrops}.
\end{prop}
\begin{proof}
Let $\gamma_{min}:[0,1]\to\mathbb{R}^2$ be a minimizer for Problem~\ref{problemdrops}
with $\gamma_{min}(0)=\gamma_{min}(1)$ placed at the origin of $\mathbb{R}^2$.
In particular 
this curve satisfies the Euler Lagrange equation $2\partial_s^2k+k^3-k=0$ in $(0,1)$, it is smooth and $k(0)=\kappa(1)=0$
(see~\cite[Proposition 4.2]{daplu}).
We symmetrize the curve with respect to the origin, and so we produce a
symmetric double drop $\mathcal{SD}_{min}$, which is a minimizer 
of $F$ among double drops, symmetric double drops and 
closed regular  curves 
 $\varphi:\mathbb{S}^1\to\mathbb{R}^2$ of class $H^2$
in which  there exists at least two points 
$s_1\neq s_2\in \mathbb{S}^1$ such that $\gamma(s_1)=\gamma(s_2)$. 
We can parametrize $\mathcal{SD}_{min}:[-1,1]\to\mathbb{R}^2$ as follow:
\begin{equation*}
\mathcal{SD}_{min}(t)=
\begin{cases}
\gamma_{min}(t)\quad\quad\;\,\,\text{for}\;t\in[0,1]\\
-\gamma_{min}(-t)\quad\text{for}\;t\in[-1,0)\,.\\
\end{cases}
\end{equation*} 
Then  $\mathcal{SD}_{min}$ satisfies the Euler Lagrange equation $2\partial_s^2k+k^3-k=0$ in $(-1,0)$ and in $(0,1)$, it  is piecewise smooth and
globally $C^3$. By continuity  $\mathcal{SD}_{min}$ satisfies the Euler Lagrange equation $2\partial_s^2k+k^3-k=0$ also at the junction point 
and thus it is globally an elastica.
Notice that this property is shared by all minimizers in the class of symmetric double drops.
The unique closed elastica with at least one
self-intersection is the ``Figure Eight" $\mathcal{F}$~\cite[Theorem 0.1 (a)]{langsing1}.
Hence the minimizer among symmetric double drop is unique and coincides 
with the ``Figure Eight" $\mathcal{F}$.
We stress the fact that we took an arbitrary minimizer among drops
$\gamma_{min}$ and with our reflection argument we produced the
``Figure Eight". As a consequence we conclude that also the minimizer for 
Problem~\ref{problemdrops} is unique and it is  one of the two drops of 
$\mathcal{F}$.
The energy of the minimizer among drops $\mathcal{D}_{min}$ or curves with at least one self--intersection $\mathcal{S}_{min}$
is twice the energy of the unique minimizer among drops $\gamma_{min}$.
Thus each one of the two drops of $\mathcal{D}_{min}$ and $\mathcal{S}_{min}$
is $\gamma_{min}$. The  two drops are attached at the origin is such a way that 
$\mathcal{D}_{min}$ and $\mathcal{S}_{min}$ are globally $C^1$ and hence the only possibility is the ``Figure Eight" $\mathcal{F}$ and this concludes the proof.
\end{proof}

It remains to compute the energy of the ``Figure Eight". Langer and Singer already observed in \cite{langsing} that the Euler-Lagrange equation
 of the elasticae can be integrated using Jacobi-Elliptic functions. 
In~\cite{explicitelasticae} the authors found a dynamical system 
that the components of a planar elastica parametrized by arc-length satisfy. 
From this dynamical system description one is able to find an explicit parametrization 
of the \lq\lq Figure Eight\rq\rq ~depending only on well defined parameters. 
Thanks to this representation we were able to compute in \cite[Prop.6.4]{daplu} the energy of $\mathcal{F}$.

\begin{lem}[see \cite{daplu}]\label{energiaotto}
Consider the optimal rescaling $\mathcal{F}$ of the ``Figure Eight" ,
then $$F(\mathcal{F})\approx 21.2075.$$
\end{lem}

Let us note here that the words \lq\lq optimal rescaling\rq\rq~would not be necessary in the statement once one defines the ``Figure Eight" as the unique (up to multiple coverings and isometries of $\mathbb{R}^2$) closed planar critical point of $F=F_1$ with self--intersections.

\subsection{Minimizers are Theta--networks}

We are ready to establish the main result of the paper.

\begin{thm}
There exists a minimizer of Problem~\ref{Pdoppiabolla}.
\end{thm}

\begin{proof}
In Proposition~\ref{minimogoccia} we have shown that
 the minimum of problem~\ref{problemdrops} is 
attained by one of the two drops of 
the``Figure Eight" $\mathcal{F}$. As a consequence, 
the energy of 
any ``degenerate" Theta--network $\Gamma$ is greater of equal than
the energy of $\mathcal{F}$, that is 
$\overline{F}(\Gamma)\geq \overline{F}({\mathcal{F}})\approx 21.2075$.
By~\eqref{energiadellabolla} the energy 
of the optimal rescaling of the standard double
bubble is  $F(\mathcal{B}_{\overline{r}})\approx 18,4059$.
We have
exhibit a Theta--network with strictly less energy than every 
 ``degenerate" Theta--network.
 This shows that each minimizer for $\overline{F}$ (whose existence is
established in Corollary \ref{esistenzatrenet}) is necessarily
a Theta--network. 
Since $\overline{F}$ is the relaxation of $F$ the first part of the statement follows.
\end{proof}

With an argument based on Theorem \ref{stimamodulok} 
we can show that minimizers are injective.

\begin{prop}\label{iniet}
Every minimizer of Problem~\ref{Pdoppiabolla} is an injective Theta--network.
\end{prop}
\begin{proof}
For simplicity we divide the proof into two steps.
\begin{itemize}
\item {\it{Step 1: At most one curve  has self--intersections}}

Suppose by contradiction that at least two curves, let us say
$\gamma^1$ and $\gamma^2$, 
of a minimizer $\Gamma$ have self-intersections.
Then for $i\in\{1,2\}$ we can decompose 
$\gamma^i$
as the union of two curves 
$\mathcal{D}^i$ and $\left(\gamma^i\setminus \mathcal{D}^i\right)$,
with $\mathcal{D}^i$ a drop.
By Proposition~\ref{minimogoccia} and Lemma~\ref{energiaotto}
we have that
\begin{equation*}
F(\Gamma)\geq F(\gamma^1)+F(\gamma^2)
\geq F(\mathcal{D}^1)+ F(\mathcal{D}^2)\geq F(\mathcal{F})
>F(\mathcal{B}_{\overline{r}})\,,
\end{equation*} 
hence such a $\Gamma$ is not a minimizer, a contradiction.
\item 
{\it{Step 2: Each curve is injective}}

Suppose that 
the curve $\gamma^1$
has self-intersections.
As before 
$\gamma^1$ can be written as the union of two curves 
$\mathcal{D}$
and $\left(\gamma^1\setminus \mathcal{D}\right)$,
with $\mathcal{D}$ a drop and clearly $F(\gamma^1)\geq F(\mathcal{D})$.
Thanks to Corollary~\ref{minimogoccia} we know that
the energy of $\mathcal{D}$ is greater or equal than the energy of one drop
of 
the ``Figure Eight" $\mathcal{F}$.
Moreover combining Corollary~\ref{variantGB} with Lemma~\ref{formula2c} we get
$F(\gamma^2\cup\gamma^3)\geq \frac{8\pi}{3}$.
Hence 
$$
F(\Gamma)\geq F(\mathcal{D})+F(\gamma^2\cup\gamma^3)\geq
\frac{1}{2}F(\mathcal{F})+ \frac{8\pi}{3}\approx 18.9813 >F(\mathcal{B}_{\overline{r}})\,,
$$
a contradiction.\qedhere
\end{itemize}
\end{proof}

Now that we know that each minimizer is a Theta--network it is interesting to understand which
boundary conditions a minimizer (or more generally a critical point) in the class of Theta--networks satisfies. The following result is established in \cite[Prop.4.1]{daplu}. See also \cite{baganu}.

\begin{prop}[see \cite{daplu}]
Let $\Gamma=\{\gamma^1,\gamma^2,\gamma^3\}$ be a Theta--network that is a critical point for $F$. Then each curve $\gamma^i:[0,1] \to\mathbb{R}^2$ of $\Gamma$ is $C^{\infty}$ and satisfies \eqref{EL:curve} on $(0,1)$. Moreover, at each junction the following conditions are satisfied
\begin{equation}\label{condjunc}
\sum_{i=1}^3 k^i(x)=0 \mbox{ and } \sum_{i=1}^3 \left(2 \partial_s k^i \nu^i + (k^i)^2\partial_s \gamma^i \right)(x)=0 \mbox{ at }x=0,1 \,.
\end{equation}
Here $k^i$ is the scalar curvature of $\gamma^i$ and  $\nu^i$ is the normal to $\gamma^i$, i.e. the counter-clockwise rotation of $\frac\pi2$ of the tangent $\partial_s \gamma^i$.
\end{prop}

A direct consequence of \eqref{condjunc} is that the standard double bubble of radius $r$ is not a critical point of $F$, being the second condition in \eqref{condjunc} not satisfied. Indeed, consider the double bubble given by the curves $\gamma^1,\gamma^2,\gamma^3$ as in Figure \ref{bubble} with $\gamma^1$ oriented counter-clockwise, $\gamma^2$ oriented from the right to the left and $\gamma^3$ oriented clockwise. Being the curvature constant it is clear that $\partial_s k^i \equiv 0$ for $i=1,2,3$. On the other hand $k^2\equiv 0$ whereas $k^3(x)=-k^1(x)\equiv -r^{-1}$ for all $x \in [0,1]$. Since $\partial_s \gamma^1(0)=(\frac12,\frac12 \sqrt{3})$ and $\partial_s \gamma^2(0)=(\frac12,-\frac12 \sqrt{3})$ it follows
$$ \sum_{i=1}^3 \left(2 \partial_s k^i \nu^i + (k^i)^2\partial_s \gamma^i \right)(0) = r^{-2} (1,0) \ne (0,0) \, .$$

Among the qualitative properties that a minimizer
could have, we  expect global embeddedness, i.e. that there are no intersections (except at the junction points) among the curves of a minimal network. 
It seems also plausible that one of the three curves of the minimizer is a straight line and that the minimizer is symmetric, as shown in Figure~\ref{figconjecture}. 

It would also be interesting to understand what happens if one consider the same minimization problem in $\mathbb{R}^n$, instead of $\mathbb{R}^2$.

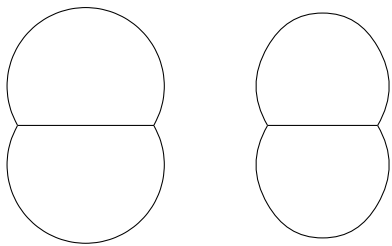
\begin{figure}[H]
\begin{center}
\begin{tikzpicture}[scale=0.17]
\draw[color=black,scale=6.15,domain=-2.09: 2.09,
smooth,variable=\t,shift={(0,0)},rotate=0]plot({1.*sin(\t r)},
{1.*cos(\t r)}) ; 
\draw[color=black,scale=6.15,domain=-2.09: 2.09,
smooth,variable=\t,shift={(0,-1)},rotate=180]plot({1.*sin(\t r)},
{1.*cos(\t r)}) ; 
\draw[scale=6.15, black]
(-0.87,-0.5)--(0.87,-0.5);
\end{tikzpicture}\qquad\quad
\begin{tikzpicture}[scale=0.17]
\draw[rotate=180, shift={(0,6.16)}, scale=6.15, black]
(0.7,-0.5)to[out= 60,in=-60, looseness=1] (0.7,0.5)
(0.7,0.5)to[out= 120,in=0, looseness=1] (0,0.93)
(-0.7,-0.5)to[out= 120,in=-120, looseness=1] (-0.7,0.5)
(-0.7,0.5)to[out= 60,in=180, looseness=1] (0,0.93);
\draw[scale=6.15, black]
(0.7,-0.5)to[out= 60,in=-60, looseness=1] (0.7,0.5)
(0.7,0.5)to[out= 120,in=0, looseness=1] (0,0.93)
(-0.7,-0.5)to[out= 120,in=-120, looseness=1] (-0.7,0.5)
(-0.7,0.5)to[out= 60,in=180, looseness=1] (0,0.93);
\draw[scale=6.15, black]
(-0.7,-0.5)--(0.7,-0.5);
\draw[white]
(-0.01,-12.3)--(0.01,-12.3);
\end{tikzpicture}
\end{center}
\caption{Left: The standard double bubble. Right: A possible minimizer. This configuration has been
(numerically) obtained by R.~N\"{u}rnberg, letting evolve the
double bubble by the $L^2$--gradient flow of the energy $F$ (see also~\cite{baganu}).}\label{figconjecture}
\end{figure}

\section{Generalized Theta--networks}

Some of the results just presented can be generalized to a certain family of $3$-networks 
of class $H^2$ whose curves meet in junctions with fixed angles, 
that we call generalized Theta--networks. This has already been observed in~\cite{daplu}.

\begin{defn}\label{genThetanetwork}
Let $\alpha_1,\alpha_2,\alpha_3 \in (0,2\pi)$ be such that $\sum_{i=1}^3 \alpha_i=2\pi$. Then a \emph{generalized $(\alpha_1,\alpha_2,\alpha_3)$ Theta--network} is a $3$--network $\Gamma=\{\gamma^1,\gamma^2,\gamma^3\}$ such that the three curves are
of class $H^2$, regular and meet at each triple junctions with fixed angles $\alpha_1$ between $\gamma^1$ and $\gamma^2$, $\alpha_2$ between $\gamma^2$ and $\gamma^3$ and, as a consequence, $\alpha_3$ between $\gamma^3$ and $\gamma^1$.
\end{defn}

Without loss in generality we may assume that 
$0 <\alpha_1 \leq\alpha_2\leq \alpha_3 <2\pi$. Also in this setting the scaling arguments apply and hence we can consider simply the functional $F$. 

\smallskip

\paragraph{Upper bound on the energy.} 

A competitor is the generalized $(\alpha_1,\alpha_2,\alpha_3)$ Theta--network 
given by two arc of circles joined by a segment. The energy of the optimal rescaling is computed
 in~\cite{daplu} and is given by
$$
F(B_{opt})
=4\sqrt{\alpha_1+\frac{\alpha_2\sin\alpha_2}{\sin\alpha_1}}
\sqrt{\alpha_1+ \alpha_2\frac{\sin\alpha_1}{\sin\alpha_2}+\sin\alpha_1}\,.
$$ 

\smallskip

\paragraph{Extension.}

In analogy to Definition~\ref{funzionalerilassato} we define $\overline{F}$ 
the extension of $F$ to $3$--networks $\Gamma$ of class $H^2$.

\begin{defn}\label{degegeneral}
A  \textit{``degenerate'' 
generalized $(\alpha_1,\alpha_2,\alpha_3)$ Theta--network} is a network  composed by two regular curves of class $H^2$,
forming angles in pairs of $\alpha_i$ and $\pi-\alpha_i$
degrees, for some $i\in\{1,2,3\}$,
and by a curve of length zero.
\end{defn}

Then the functional $\overline{F}$  is defined as follows:
$
\overline{F}(\Gamma)=
F(\Gamma)$ if $\Gamma$ is a generalized $(\alpha_1,\alpha_2,\alpha_3)$ Theta--network, 
$$\overline{F}(\Gamma)= \sum_{i=1}^2 \int_{\gamma_i} (k^i)^2+1 \,{\rm{d}}s^i$$
if $\Gamma=\{\gamma^1,\gamma^2\}$ is a ``degenerate'' 
generalized $(\alpha_1,\alpha_2,\alpha_3)$ Theta--network and equal to $
+\infty$ otherwise.

\smallskip

\paragraph{Existence of minimizers of the relaxed energy.}

To show the  compactness  $\overline{F}$ 
it was crucial the application of Theorem~\ref{stimamodulok}. 
The same arguments work also in this setting. 
In particular with a variant of Corollary~\ref{variantGB} (due to the different angles) one gets an estimate similar to~\eqref{eq:boundbelow}  from which one still can argue that only the length of one of the curves might go to zero at the limit. 

\medskip

Once the existence of minimizers for $\overline{F}$ is established, it remains to see when one can show that the minimizers are not ``degenerate'' Theta--networks. 
Reasoning as above, in~\cite[pages 112-113]{daplu} it is shown that this is the case when $\alpha_1 \leq \alpha_2 \leq \frac34 \pi$. For general angles  $(\alpha_1,\alpha_2,\alpha_3)$ this is still an open problem.


\appendix

\section{An estimate on the curvature}
Consider a continuous, piecewise $W^{2,1}$, regular
closed curve $\gamma:[0,1]\to\mathbb{R}^2$
possibly with angles (points in which the unit tangent vector does not
change in a continuous way) 
and possibly with self--intersections.
Our aim is to find a uniform estimate for the integral of the modulus of the curvature
along the curve $\gamma$ based on  the classical Gauss--Bonnet theorem.

\begin{figure}[H]
\begin{center}
\begin{tikzpicture}[scale=2.4]
\draw[shift={(1.424,-0.78)}, scale=1.5, rotate=-150]
(0,0)to[out= -45,in=135, looseness=1] (0.05,-0.05)
(0,0)to[out= -135,in=45, looseness=1] (-0.05,-0.05)
(0,-0.3)to[out= 90,in=-90, looseness=1] (0,0.003);
 \draw[shift={(1.424,-.01)}, scale=1.5, rotate=-30]
(0,0)to[out= -45,in=135, looseness=1] (0.05,-0.05)
(0,0)to[out= -135,in=45, looseness=1] (-0.05,-0.05)
(0,-0.3)to[out= 90,in=-90, looseness=1] (0,0.003);
 \draw[shift={(0.42,-0.85)}, scale=1.5, rotate=-70]
(0,0)to[out= -45,in=135, looseness=1] (0.05,-0.05)
(0,0)to[out= -135,in=45, looseness=1] (-0.05,-0.05)
(0,-0.3)to[out= 90,in=-90, looseness=1] (0,0.003);
\draw[shift={(0.42,-1.152)}, scale=1.5, rotate=-110]
(0,0)to[out= -45,in=135, looseness=1] (0.05,-0.05)
(0,0)to[out= -135,in=45, looseness=1] (-0.05,-0.05)
(0,-0.3)to[out= 90,in=-90, looseness=1] (0,0.003);

\draw[black!60!white, thick, shift={(0,-1)}]
(0,0)to[out= 20,in=-90, looseness=1] (0.7,0.45)
(0,0)to[out= 160,in=-90, looseness=1] (-0.9,0.65)
(0.7,0.45)to[out= 90,in=-130, looseness=1] (1.2,0.6)
(1.2,0.6)to[out= -60,in=180, looseness=1] (1.8,0)
(1.8,0)to[out= 0,in=-120, looseness=1] (2.6,0.3);
\draw[black!60!white, thick,, shift={(2,-1)}, scale=1, rotate=-90]
(0,0)to[out= -135,in=45, looseness=1] (-0.05,-0.05)
(0,0)to[out= -45,in=135, looseness=1] (0.05,-0.05);
\draw[black!60!white, thick, shift={(-0.88,-0.5)}, scale=1, rotate=-160]
(0,0)to[out= -135,in=45, looseness=1] (-0.05,-0.05)
(0,0)to[out= -45,in=135, looseness=1] (0.05,-0.05);

\fill[black](0,-1) circle (0.5pt);
\fill[black](1.2,-0.4) circle (0.5pt);
\path[shift={(0.75,-0.56)}] 
 (0.13,0)[right] node{\tiny{$\gamma(x_j)$}};
\path[shift={(-0.5,-1.1)}] 
 (0.13,0)[right] node{\tiny{$\gamma(x_i)$}};
 
 \path[shift={(0.75,-0.56)}] 
 (0.55,0.15)[right] node{\tiny{$\tilde{\theta}_j$}};
\path[shift={(0.1,-1)}] 
(0.13,0)[right] node{\tiny{$\tilde{\theta}_i$}};
 
 \draw[color=black,scale=0.05,domain=1.3: 1.9,
smooth,variable=\t,shift={(0,-20)},rotate=0]plot({2.*sin(\t r)},
{2.*cos(\t r)}) ; 
\draw[black, shift={(0.1,-1)}, scale=0.3, rotate=0]
(0,0)to[out= -135,in=45, looseness=1] (-0.05,-0.05)
(0,0)to[out= -45,in=135, looseness=1] (0.05,-0.05);

 \draw[color=black,scale=0.05,domain=0.5: 2.7,
smooth,variable=\t,shift={(24,-8)},rotate=0]plot({2.*sin(\t r)},
{2.*cos(\t r)}) ; 
\draw[ shift={(1.3,-0.4)}, scale=0.3, rotate=180]
(0,0)to[out= -135,in=45, looseness=1] (-0.05,-0.05)
(0,0)to[out= -45,in=135, looseness=1] (0.05,-0.05);
\path[shift={(-1,-0.5)}] 
 (0.13,0)[right] node{$\gamma$};
\end{tikzpicture}
\end{center}\caption{A curve with external angles $\tilde{\theta}_i$ and $\tilde{\theta}_j$.} 
\end{figure}
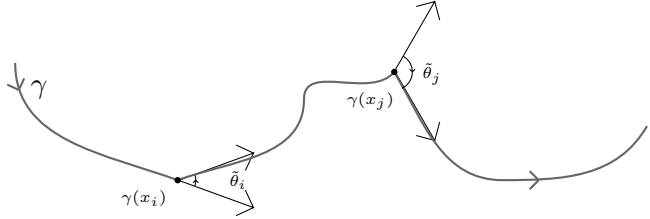

Let $0=x_0<x_1<x_2\ldots<x_N
=1$ be a partition of $[0,1]$
such that $\gamma$ is of class $W^{2,1}$ on each interval $[x_{i-1},x_{i}]$, $i \in \{1,..,N\}$. 
We call  $\gamma(x_i)$  vertices.
In the following, for each $i$, $\theta_i$ denotes the (positive) 
angle between the unit tangent vectors
$\tau(x_i^-)=\partial_s \gamma(x_i^-)$ and $\tau(x_i^+)= \partial_s \gamma(x_i^+)$. 
More precisely,
at each vertex we fix $\theta_i\in[0,\pi]$, $i \in \{1,..,N\}$, such that
\begin{equation}\label{cosangolo}
\cos\left(\theta_i\right)=
\frac{\left\langle\gamma'(x_i^-),\gamma'(x_i^+)\right\rangle}{\vert\gamma'(x_i^-) \vert \,
\vert\gamma'(x_i^+) \vert}\,,
\end{equation}
with the usual convention that $x_N^+:=x_{0}^+$.
Notice that the angles $\theta_i$ do not depend on the parametrization and the orientation.

\begin{thm}\label{stimamodulok}
Consider a continuous, piecewise $W^{2,1}$, regular 
closed curve $\gamma:[0,1]\to\mathbb{R}^2$, 
possibly with 
angles $\theta_i$, $i\in\{1,\ldots, N\}$, at the vertices $\gamma(x_i)$. 
Then
\begin{equation}\label{general}
\int_{\gamma}\vert k \vert \,{\rm{d}}s\geq 2\pi-\sum_{i=1}^N \theta_i\,.
\end{equation}
\end{thm}
\begin{proof}
We first suppose that 
the closed curve $\gamma:[0,1]\to\mathbb{R}^2$ is piecewise $C^\infty$, regular,
possibly with 
angles $\theta_i$, $i\in\{1,\ldots, N\}$,
at the vertices $\gamma(x_i)$ and possibly with finitely many self--intersection
but not at the vertices $\gamma(x_i)$.
Then the claim in the general case follows by approximation, 
combining the fact that $C^k$ is dense in $W^{2,1}$
and the density of the considered class of curve into rectifiable curves. This approximation argument is given in detail in Appendix \ref{sect:approxi}.

\noindent
We prove the statement by induction on the number of self--intersections of $\gamma$. 

\noindent
Suppose first that the curve $\gamma$ has no self-intersections and hence 
$\gamma$ is a simple closed curve. The claim then follows from the classical Gauss--Bonnet 
Theorem~\cite[pag. 269]{docarmo}. 
Let us shortly discuss the ideas. Without loss of generality 
we can assume that the curve is positively oriented.
For all $\theta_i\in[0,\pi)$, $i\in \{1, ..,N\}$, we define the external
angles $\tilde{\theta}_i$ as follows:
$\tilde{\theta}_i=\theta_i$ if $\gamma'(x_i^+)$ is obtained by a counterclockwise
rotation of $\gamma'(x_i^-)$
and $\tilde{\theta}_i=-\theta_i$ otherwise (and
we refer to~\cite[pages 265--267]{docarmo} for the more delicate definition of $\tilde{\theta}_i$
in the case $\theta_i=\pi$). Then, $\tilde{\theta}_i\in[-\pi,\pi]$.
The classical Gauss--Bonnet 
Theorem~\cite[pag. 269]{docarmo} yields
\begin{equation*}
\int_{\gamma} k  \,{\rm{d}}s= 2\pi-\sum_{i=1}^N \tilde{\theta}_i\,.
\end{equation*}
Hence
\begin{equation*}
\int_{\gamma} \vert k \vert \,{\rm{d}} s\geq 
\int_{\gamma} k  \,{\rm{d}}s= 2\pi-\sum_{i=1}^N \tilde{\theta}_i\geq 
2\pi-\sum_{i=1}^N \theta_i\,.
\end{equation*}

\noindent
Let us assume 
now that Formula~\eqref{general} holds true for every curve with at most $n$
self--intersections and consider a 
curve $\gamma$ with $n+1$ self--intersections.
Then we can
decompose $\gamma$ as union of  two curves,
$\gamma^1,\gamma^2:[0,1]\to\mathbb{R}^2$, 
each with at most $n$ self--intersections, so that:
$\gamma^1$ has
$M$  angles $\theta_i$ ($M\leq N$) 
and an angle $\beta$ (with $0\leq \beta\leq \pi$) 
at $\gamma^1(0)=\gamma^1(1)$
and the curve $\gamma^2:[0,1]\to\mathbb{R}^2$ has
$N-M$  angles and an angle $\beta$ 
(with $0\leq\beta\leq \pi$) 
at $\gamma^2(0)=\gamma^2(1)$.
Then by the induction assumption
we get 
\begin{align*}
\int_{\gamma}\vert k \vert \,{\rm{d}}s=&\int_{\gamma^1}\vert k \vert \,{\rm{d}}s+\int_{\gamma^2}\vert k \vert \,{\rm{d}}s
\geq 2\pi-\sum_{i=1}^M \theta_i-\beta+ 2\pi-\sum_{i=M+1}^N \theta_i-\beta\\
=&4\pi-\sum_{i=1}^N \theta_i -2\beta\geq 2\pi-\sum_{i=1}^N \theta_i \,.
\end{align*}

\end{proof}

\begin{cor}\label{variantGB1curve}
Consider a regular curve $\gamma:[0,1]\to\mathbb{R}^2$ 
of class $W^{2,1}$, 
such that $\gamma(0)=\gamma(1)$.
Then 
$$
\int_{\gamma}\vert k \vert \,{\rm{d}}s\geq \pi\,.
$$
\end{cor}
\begin{proof}
Since $\gamma$ has only an angle 
with arbitrary amplitude $\theta\in [0,\pi]$ the claim follows directly from~\eqref{general}.
\end{proof}

\begin{cor}\label{variantGB}
Consider a regular closed curve $\gamma:[0,1]\to\mathbb{R}^2$ 
piecewise
of class $W^{2,1}$,
with two equal angles $\theta_i=\frac{\pi}{3}$.
Then 
$$
\int_{\gamma^1\cup\gamma^2}\vert k \vert \,{\rm{d}}s\geq \frac{4\pi}{3}\,.
$$
\end{cor}
\begin{proof}
Also in this case the statement is a direct consequence of~\eqref{general}.
\end{proof}


\section{An approximation lemma}\label{sect:approxi}

In this section we give the details of the approximation argument used in the proof of Theorem \ref{stimamodulok}. The concept of \textit{transversality} of curves turns out to be crucial in our argument.

\begin{defn}
Consider two curves $\gamma,\tilde{\gamma}:[0,1]\to\mathbb{R}^2$
of class $C^k$ (with $k=1,2,\ldots$). We say that
\begin{enumerate}
\item A curve $\gamma$
is \emph{self--transversal} if at each self--intersection
$\gamma(x)=\gamma(\tilde{x})$ with $x\neq \tilde{x}$
the derivatives $\gamma'(x)$ and $\gamma'(\tilde{x})$
are linearly independent.
\item Two curves $\gamma,\tilde{\gamma}$ are \emph{transversal} if at each intersection
$\gamma(x)=\tilde{\gamma}(\tilde{x})$ the derivatives
$\gamma'(x)$ and $\tilde{\gamma}'(\tilde{x})$ are linearly independent.
\item Consider a piecewise $C^k$  curve  
$\gamma:[0,1]\to\mathbb{R}^2$ and 
let $0=x_0<x_1<x_2\ldots<x_N=1$ be a partition of $[0,1]$
such that $\gamma_{\vert [x_{i-1},x_{i}]}$  is of class $C^k$ for $i \in \{1,..,N\}$.
Then $\gamma$ is \textit{self--transversal} if 
 $\gamma_{\vert [x_{i-1},x_{i}]}$ 
is self--transversal for every $i\in\{1,\ldots, N\}$ and the curves $\gamma_{\vert [x_{i-1},x_{i}]}$ 
and $\gamma_{\vert [x_{j-1},x_{j}]}$
are transversal for $i \ne j$, $i,j \in \{1,..,N\}$.
\end{enumerate}
\end{defn}

\begin{rem} 
Notice that if a  curve $\gamma$
is  \emph{regular}, then
every point in the image of the curve has a finite number of preimages.
Moreover, a compactness argument yields that 
 every \emph{self--transversal} curve has a finite number of self--intersections.
Similarly, if two curves $\gamma:[0,1]\to\mathbb{R}^2$
and $\tilde{\gamma}:[0,1]\to\mathbb{R}^2$
are regular and transversal, then they have a finite number of
intersections and the number of couples $(x,\tilde{x})$
with $x,\tilde{x}\in[0,1]$ and 
$\gamma(x)=\tilde\gamma(\tilde x)$ is finite.
\end{rem}

In the next result we describe how to approximate 
a piecewise $C^k$  
regular 
closed curve, 
$\gamma:[0,1]\to\mathbb{R}^2$,
possibly with angles $\theta_i$ ($i\in\{1,\ldots, N\}$)
at the vertices $\gamma(x_i)$, 
with
closed curves 
$\tilde{\gamma}:[0,1]\to\mathbb{R}^2$ 
which are
piecewise $C^\infty$, regular,
and with finitely many self--intersections
(but non at the vertices). 
%
%
%
%
\begin{lem}[Approximation]\label{approx}
Let $k = 1,2,\ldots$,
consider a  continuous closed curve  
$\gamma:[0,1]\to\mathbb{R}^2$ 
and 
let $0=x_0<x_1<x_2\ldots<x_N=1$ be a partition of $[0,1]$
such that 
$\gamma$ is of class $C^k$ and regular
on each interval $[x_{i-1},x_{i}]$.
Then
 for every $\varepsilon>0$ there exists $\widetilde{\gamma}:[0,1]\to\mathbb{R}^2$
 an approximating 
continuous  closed curve which is
   of class $C^\infty$ and regular on each interval $[x_{i-1},x_{i}]$
and 
such that
\begin{enumerate}
\item $\Vert \gamma-\widetilde{\gamma}\Vert_{C^k[x_{i-1},x_{i}]}\leq \varepsilon$;
\item $\widetilde{\gamma}$ is self--transversal;
\item $\widetilde{\gamma}$ has no self--intersections at the vertices $\widetilde{\gamma}(x_i)$, $i \in\{1,..,N\}$.
\end{enumerate}
\end{lem}

\begin{proof}
{\it{Step 1: Construction of an approximating curve without self--intersections at the vertices}}\\ Since there are only finitely many vertices it is sufficient to consider each one separately. 
Suppose that $\gamma$ has self--intersections at the vertex $\gamma(x_1)$.
Since $\gamma$ is regular, $\partial_x \gamma(x_1^-)\neq 0$.
Then at least one component of $\partial_x \gamma (x_1^-)$ is different from zero
and, by continuity, there exists $\eta_1>0$ such that this component of $\partial_x \gamma(x)$
is different from zero
for all $x\in [x_1-\eta_1, x_1)$ and also satisfying $x_1-\eta_1\geq \frac12 x_1$.
Similarly $\partial_x \gamma(x_1^+)\neq 0$, hence we can find $\eta_2>0$ such that $x_1+\eta_2 \leq \frac12(x_1+x_2)$
 and 
such that  $\gamma(x)\neq \gamma(x_1)$ for all $x\in [x_1, x_1+\eta_2]$.\\
Let $\eta$ be equal to the smallest number between $\eta_1$ and $\eta_2$ and take 
$\psi:[0,1]\to[0,1]$ of class $C^\infty$ with support contained in  
$((x_1-\eta, x_1+\eta)$ and 
equal to one in $(x_1-\frac12 \eta, x_1+ \frac12 \eta)$ and otherwise taking values between $0$ and $1$.
Consider now the curve $\hat\gamma=\gamma+v\psi$ where 
the vector $v$ in $\mathbb{R}^2$ should be chosen such that 
\begin{enumerate}
\item $\vert v\vert \,\Vert \psi\Vert_{C^k([x_1-\eta,x_{1}+\eta])}\leq \frac14 \varepsilon$;
\item $\hat\gamma(x_0)\notin \mathrm{Im}(\gamma)$;
\item for every $x\in[x_1-\eta,x_1]$ (respectively in $[x_1,x_1+\eta]$)
a component of $\partial_x \hat{\gamma}$ (the same of $\partial_x \gamma$)
 is different from zero.
\end{enumerate}
Condition $1.$ and $3.$ are satisfied by choosing $v$ small enough
and as $\mathcal{H}^2(\mathrm{Im}(\gamma))=0$ all conditions can be achieved. 
Then 
\begin{itemize}
\item condition $1.$ implies that 
$\Vert \gamma-\hat{\gamma}\Vert_{C^k([x_{i-1},x_{i}])}\leq \frac14 \varepsilon$ for $i=1,2$;
\item condition $2.$ implies that 
$\hat\gamma(x_1)\neq \hat\gamma(x)$ with 
$x\in ([0,1]\setminus [x_1-\eta,x_1+\eta])$;
\item condition $3.$ implies that 
$\hat\gamma(x_1)\neq \hat\gamma(x)$ with 
$x\in  [x_1-\eta,x_1+\eta]$.
\end{itemize}
Hence by construction $\hat{\gamma}$ has a vertex in $x_1$ but no self-intersections at this vertex.

\noindent 
If $\hat{\gamma}$ has self--intersections at the vertices 
$\hat{\gamma}(x_i)$, $i\in\{2,\ldots,N\}$, one has to repeat the approximation procedure
described above at each vertex. For simplicity we still call the new approximating curve so obtained $\hat{\gamma}$.

{\it{Step 2: Construction of a self--transversal approximating curve}}\\
For every $i\in\{1,\ldots,N\}$
let $\hat{\gamma}^i$ be the restriction of  the curve $\hat{\gamma}$ 
to the interval $[x_{i-1},x_{i}]$. For every $\varepsilon>0$  and  for every $i\in\{1,\ldots, N\}$ there exists 
a curve $\widetilde{\gamma}^i:[x_{i-1},x_{i}]\to\mathbb{R}^2$
of class $C^\infty$, self--transversal, such that for all $j\in\{1,\ldots,N\}$ with $j< i$
the curves
$\widetilde{\gamma}^i$ and $\widetilde{\gamma}^j$ are transversal 
and 
$\Vert \hat{\gamma}^i-\widetilde{\gamma}^i\Vert_{C^k([x_{i-1},x_{i}])}\leq \varepsilon/2$.
The key lemmas are the following:
the family of regular curves 
of class $C^k$ and self--transversal
is open and dense in $(C^k,\Vert\cdot\Vert_{C^k})$
(see Theorem 1.1 of Chapter 1,
Theorem 2.12 of Chapter 2 and Exercise 2 of Section 2 in Chapter 3
in~\cite{Hirsch})
and the family of all regular curves 
of class $C^k$ and transversal to
a given curve  is open and dense in $(C^k,\Vert\cdot\Vert_{C^k})$
(see Theorem 2.1(b) in Chapter 3 in~\cite{Hirsch}).\\
The desired approximating curve $\widetilde{\gamma}$ is 
given by the union of all the $\widetilde{\gamma}^i$
and therefore is  continuous,
piecewise $C^\infty$ curve and regular.
\end{proof}

\bibliographystyle{amsplain}
\bibliography{networkelastici_final}

\end{document}